\documentclass[12pt,reqno]{amsart}
\usepackage{amsmath,amsthm,amsfonts,amssymb,times}
\usepackage{verbatim}
\usepackage{url}
\usepackage{xcolor}
\setbox0=\hbox{$+$}
\newdimen\plusheight
\plusheight=\ht0
\def\+{\;\lower\plusheight\hbox{$+$}\;}

\setbox0=\hbox{$-$}
\newdimen\minusheight
\minusheight=\ht0
\def\-{\;\lower\minusheight\hbox{$-$}\;}

\setbox0=\hbox{$\cdots$}
\newdimen\cdotsheight
\cdotsheight=\plusheight
\def\cds{\lower\cdotsheight\hbox{$\cdots$}}

\renewcommand{\Re}{\operatorname{Re}}
\renewcommand{\Im}{\operatorname{Im}}

\makeatletter
\def\leqalignno#1{\displ@y \tabskip\z@ plus\@ne fil
  \halign to\displaywidth{\hfil$\@lign\displaystyle{##}$\tabskip\z@skip
    &$\@lign\displaystyle{{}##}$\hfil\tabskip\z@ plus\@ne fil
    &\kern-\displaywidth\rlap{$\@lign\hbox{\rm##}$}\tabskip\displaywidth\crcr
    #1\crcr}}
\makeatother

\newcommand{\eb}{\begin{equation}}
\newcommand{\ee}{\end{equation}}
\renewcommand{\Im}{\operatorname{Im}}
\newcommand{\df}{\dfrac}
\newcommand{\tf}{\tfrac}

\renewcommand{\Re}{\operatorname{Re}}
\renewcommand{\Im}{\operatorname{Im}}

 \renewcommand{\a}{\alpha}
\renewcommand{\b}{\beta}
\newcommand{\e}{\epsilon}
\renewcommand{\d}{{\delta}}

\renewcommand{\o}{{\omega}}

\newcommand{\Q}{\mathbb{Q}}

\renewcommand{\Re}{\text{Re}}
\renewcommand{\Im}{\text{Im}}

\renewcommand{\(}{\left\(}
\renewcommand{\)}{\right\)}
\renewcommand{\[}{\left\[}
\renewcommand{\]}{\right\]}

\renewcommand{\pmod}[1]{\,(\textup{mod}\,#1)}
\numberwithin{equation}{section}
 \theoremstyle{plain}

\newtheorem{theorem}{Theorem}[section]
\numberwithin{equation}{section}
\theoremstyle{plain}

\newtheorem{lemma}[theorem]{Lemma}
\newtheorem{corollary}[theorem]{Corollary}

\newtheorem{definition}[theorem]{Definition}

\allowdisplaybreaks

\begin{document}

\title[Finite Trigonometric Sums]
{Evaluations and  Relations for Finite Trigonometric Sums}
\author{Bruce C.~Berndt, Sun Kim, Alexandru Zaharescu}
\address{Department of Mathematics, University of Illinois, 1409 West Green
Street, Urbana, IL 61801, USA} \email{berndt@illinois.edu}
\address{Department of Mathematics, and Institute of Pure and Applied Mathematics, Jeonbuk National
University, 567 Baekje-daero, Jeonju-si, Jeol labuk-do 54896, Republic of Korea}
\email{sunkim@jbnu.ac.kr}
\address{Department of Mathematics, University of Illinois, 1409 West Green
Street, Urbana, IL 61801, USA; Institute of Mathematics of the Romanian
Academy, P.O.~Box 1-764, Bucharest RO-70700, Romania}
\email{zaharesc@illinois.edu}

\begin{abstract}
Several methods are used to evaluate  finite trigonometric sums.  In each case, either the sum had not previously been evaluated, or it had been evaluated, but only by analytic means, e.g., by complex analysis or modular transformation formulas.   We establish both reciprocity and three sum relations for trigonometric sums. Motivated by certain sums that we have evaluated, we add coprime conditions to the summands and thereby define analogues of Ramanujan sums, which we in turn evaluate.  One of these analogues leads to a criterion for the Riemann Hypothesis, analogous to the Franel--Landau criterion.
\end{abstract}


\maketitle

\section{Introduction}
Evaluations of finite trigonometric sums arise in many contexts in mathematics and physics.  The paper \cite{BY} by the first author and B.~P.~Yeap contains a total of 89 references in which evaluations arise in many diverse areas within mathematics and physics.

  In our recent paper \cite{BKZ1}, we devised two methods to evaluate trigonometric sums.   First, we developed new approaches in contour integration, which  extended ideas that two of us made in \cite{bz}.     Second, we devised a new method for evaluating trigonometric sums.

  In the present paper, we continue our examinations of finite trigonometric sums that we made in \cite{BKZ1} and \cite{BKZ2}.  In particular, we further develop the aforementioned  approach in \cite{BKZ1}.   However, we also introduce new ideas, which we relate below.

  First, in Section \ref{section1}, we evaluate several finite sums for which only analytic proofs were previously available.

Second, recall that Ramanujan sums (defined in \eqref{Ramanujansum}) have a co-prime requirement on the indices of summation.  Our evaluations of several sine sums in Section \ref{section1}  motivated us to define certain analogues of Ramanujan sums in Section \ref{ramanujan}. More precisely, we take certain trigonometric sums from Section \ref{section1} and add a co-prime condition to the summands.  These sums have apparently not been studied in the literature.  Remarkably, we are able to find elegant evaluations of these `new' sums.

The Frenel--Landau criterion for the Riemann Hypothesis can be readily established using Ramanujan sums.  We use one of the aforementioned analogues of Ramanujan sums to establish a new criterion for the Riemann Hypothesis.  Our new criterion has two interesting and different features that are not found in the Frenel--Landau criterion.

Third, the coefficients $a_n$ of finite trigonometric sums evaluated in the literature are generally of three types: $a_n=1, a_n=(-1)^n$, and $a_n = \chi(n)$, where $\chi$ is a Dirichlet character.  In Theorem \ref{2-T3} in Section \ref{section2}, we provide a general theorem for which the coefficients $a_n$ are $\pm1$, but they do not fall under the umbrellas of the three aforementioned types.

As remarked above, in \cite{BKZ1} we introduced a new method for the evaluation of finite trigonometric sums.  Our  method  yielded evaluations and identities for which only analytic proofs were previously available.  This study continues in the present paper, but now with a concentration on reciprocity theorems and identities involving three trigonometric sums.
 In the transformation formulae of the Dedekind eta-function, and in many generalizations and/or analogues of the Dedekind eta function, sums frequently arise that have representations as finite cotangent sums.  Generally, they satisfy reciprocity theorems and, in some cases, three term relations, in the terminology of the classical literature.   (It seems more appropriate to instead call these identities `three sum' relations, and so we shall do so in the sequel.) These reciprocity theorems and three sum relations are normally derived from the transformation formulas themselves or by using the residue theorem.  In Section \ref{section3}, we provide different, and perhaps shorter, proofs of reciprocity theorems and three sum relations for several cotangent sums, In particular, we establish a reciprocity theorem for a modified Dedekind sum associated with Eisenstein series, which was  introduced by the first author in \cite{Berndt0, Berndt}.


\section{Trigonometric Sums Initially Arising from Ramanujan's Theory of Theta Functions}\label{section1}
In the section, we give elementary proofs for the evaluations (1.5)--(1.10)  in \cite{HRJ}, which were proved there using theta functions.

\begin{theorem}[Equations (1.5) and (1.6) in \cite{HRJ}]\label{thm1.2}
If $k$ is an odd positive integer $\geq 3$, then
\begin{equation}\label{(1.5)}
\sum_{j=1}^{\frac{k-1}{2}}(-1)^{j-1}\sin\Big(\frac{(2j-1)\pi}{2k}\Big)=\frac{(-1)^{\frac{k+1}{2}}}{2},
\end{equation}
and
\begin{equation}\label{(1.6)}
\sum_{j=1}^{\frac{k-1}{2}}(-1)^{j-1}\csc\Big(\frac{(2j-1)\pi}{2k}\Big)=\frac{k+(-1)^{\frac{k+1}{2}}}{2}.
\end{equation}
\end{theorem}

\begin{proof} Note that
\begin{equation*}
\sum_{j=1}^{\frac{k-1}{2}}(-1)^{j-1}\sin\Big(\frac{(2j-1)\pi}{2k}\Big)
=\frac12\sum_{j=1}^{k}(-1)^{j-1}\sin\Big(\frac{(2j-1)\pi}{2k}\Big)-\frac12(-1)^{\frac{k-1}{2}}.
\end{equation*}
With
\begin{align*}
\sum_{j=1}^{k}(-1)^{j-1}\sin\Big(\frac{(2j-1)\pi}{2k}\Big)&=\Im\left\{\sum_{j=1}^{k}(-1)^{j-1}e^{\frac{(2j-1)\pi i}{2k}}\right\} \\
&=\Im\left\{\frac{e^{\frac{\pi i}{2k}}\big(1-(-e^{\frac{\pi i}{k}})^k \big)}{1+e^{\frac{\pi i}{k}}}\right\}=0,
\end{align*}
we obtain \eqref{(1.5)}.

Similarly, we have
\begin{equation*}
\sum_{j=1}^{\frac{k-1}{2}}(-1)^{j-1}\csc\Big(\frac{(2j-1)\pi}{2k}\Big)
=\frac12\sum_{j=1}^{k}(-1)^{j-1}\csc\Big(\frac{(2j-1)\pi}{2k}\Big)-\frac12(-1)^{\frac{k-1}{2}}.
\end{equation*}
Now, observe that
\begin{align}\label{1-1}
&\sum_{j=1}^{k}(-1)^{j-1}\csc\Big(\frac{(2j-1)\pi}{2k}\Big)=
2i\sum_{j=1}^{k}(-1)^{j-1}\frac{e^{\frac{(2j-1)\pi i}{2k}}}{e^{\frac{(2j-1)\pi i}{k}}-1} \notag\\
&=-i\sum_{j=1}^{k}(-1)^{j-1}e^{\frac{(2j-1)\pi i}{2k}}\sum_{t=0}^{k-1}e^{\frac{(2j-1)\pi i}{k}t} \notag\\
&=-i\sum_{t=0}^{k-1}e^{-\frac{(2t+1)\pi i}{2k}}\sum_{j=1}^{k}(-1)^{j-1}e^{\frac{(2t+1)j\pi i}{k}}.
\end{align}
Using the identity 
\begin{equation*}
\sum_{j=1}^{k}(-1)^{j-1}e^{\frac{(2t+1)j\pi i}{k}}=
\begin{cases}
-k, & \text{if} ~ t=\frac{k-1}{2}, \\
0, & \text{otherwise},
\end{cases}
\end{equation*}
 in \eqref{1-1} we obtain
\begin{equation*}
\sum_{j=1}^{k}(-1)^{j-1}\csc\Big(\frac{(2j-1)\pi}{2k}\Big)=-i(-i)(-k)=k.
\end{equation*}
This completes the proof of \eqref{(1.6)}.
\end{proof}

The following identities were also proved in \cite{BKZ2} via contour integration.

\begin{theorem}[Equations (1.7)--(1.10) in \cite{HRJ}]\label{HRJ}
 $$   $$
\begin{enumerate}
\item Let $n$ be  an even positive number, and let $ j \equiv 2\pmod4$, where $(\tf{j}{2}, \tf{n}{2})=1$.  Then
\begin{equation}\label{(1.9)}
\sum_{k=1}^{\tf{n}{2}-1} \df{\sin\left(\dfrac{ (j-1)k\pi}{n}\right)\sin\left(\dfrac{ (j+1)k\pi}{n}\right)}
{\sin^2\left(\dfrac{ k\pi}{n}\right)\sin^2\left(\dfrac{jk\pi}{n}\right)}=\df{n^2-4}{12}.
\end{equation}

\item Let $n$ denote an even positive integer, and let $ j \equiv 2\pmod4$, where $(\tf{j}{2}, \tf{n}{2})=1$.  Then
\begin{equation}\label{(1.8)}
\sum_{k=0}^{\tf{n}{2}-1} \df{\sin\left(\dfrac{ (j-1)(2k+1)\pi}{2n}\right)\sin\left(\dfrac{ (j+1)(2k+1)\pi}{2n}\right)}
{\sin^2\left(\dfrac{ (2k+1)\pi}{2n}\right)\sin^2\left(\dfrac{j(2k+1)\pi}{2n}\right)}=\df{n^2}{4}.
\end{equation}

\item  Let $n$ be  an odd positive number, and let $ j$ be an even positive integer with  $(j,n)=1$.  Then
\begin{equation}\label{(1.7)}
\sum_{k=0}^{\tf{n-3}{2}} \df{\sin\left(\dfrac{ (j-1)(2k+1)\pi}{2n}\right)\sin\left(\dfrac{ (j+1)(2k+1)\pi}{2n}\right)}
{\sin^2\left(\dfrac{ (2k+1)\pi}{2n}\right)\sin^2\left(\dfrac{j(2k+1)\pi}{2n}\right)}=\df{n^2-1}{3}.
\end{equation}

\item Let $n$ be an odd positive integer, and let $j$ be a positive integer such that $(j,n)=1.$  Then
\begin{equation}\label{(1.10)}
\sum_{k=0}^{\tf{n-3}{2}} \df{\sin\left(\dfrac{ (j-1)(2k+1)\pi}{n}\right)\sin\left(\dfrac{ (j+1)(2k+1)\pi}{n}\right)}
{\sin^2\left(\dfrac{ (2k+1)\pi}{n}\right)\sin^2\left(\dfrac{j(2k+1)\pi}{n}\right)}=0.
\end{equation}
\end{enumerate}
\end{theorem}

\begin{proof}[Proof of \eqref{(1.9)}] Using the product formula for sine, we find that
\begin{align}\label{1.9-1}
&\sum_{k=1}^{\tf{n}{2}-1} \frac{\sin\left(\frac{(j-1)k\pi}{n}\right)\sin\left(\frac{ (j+1)k\pi}{n}\right)}
{\sin^2\left(\frac{ k\pi}{n}\right)\sin^2\left(\frac{jk\pi}{n}\right)}=
\frac12\sum_{k=1}^{\tf{n}{2}-1} \frac{\cos\left(\frac{2k\pi}{n}\right)-\cos\left(\frac{2jk\pi}{n}\right)}
{\sin^2\left(\frac{ k\pi}{n}\right)\sin^2\left(\frac{jk\pi}{n}\right)} \notag\\
&=\sum_{k=1}^{\tf{n}{2}-1} \frac{\sin^2\left(\frac{jk\pi}{n}\right)-\sin^2\left(\frac{k\pi}{n}\right)}
{\sin^2\left(\frac{ k\pi}{n}\right)\sin^2\left(\frac{jk\pi}{n}\right)}
=\sum_{k=1}^{\tf{n}{2}-1}\frac{1}{\sin^2\left(\frac{ k\pi}{n}\right)}-
\sum_{k=1}^{\tf{n}{2}-1}\frac{1}{\sin^2\left(\frac{jk\pi}{n}\right)}.
\end{align}
Now, recall \cite[Corollary 8.6]{BKZ1}
\begin{align}\label{C8.6}
\sum_{k=1}^{n-1}\frac{1}{\sin^2\left(\frac{k\pi}{n}\right)}=\frac{n^2-1}{3}.
\end{align}
From \eqref{C8.6}, we can easily deduce that
\begin{align}\label{1.9-2}
\sum_{k=1}^{\tf{n}{2}-1}\frac{1}{\sin^2\left(\frac{k\pi}{n}\right)}=\frac12\left(\frac{n^2-1}{3}-1\right)
=\frac{n^2-4}{6}.
\end{align}
If we set $j'=j/2$ and $n'=n/2,$ then since $(j',n')=1,$ we have
\begin{align}\label{1.9-3}
\sum_{k=1}^{\tf{n}{2}-1}\frac{1}{\sin^2\left(\frac{jk\pi}{n}\right)}=
\sum_{k=1}^{n'-1}\frac{1}{\sin^2\left(\frac{j'k\pi}{n'}\right)}=
\sum_{k=1}^{n'-1}\frac{1}{\sin^2\left(\frac{k\pi}{n'}\right)}=\frac{(n')^2-1}{3}=\frac{n^2-4}{12}.
\end{align}
Putting \eqref{1.9-2} and \eqref{1.9-3} into \eqref{1.9-1}, we complete the proof of \eqref{(1.9)}.

{\it Proof of \eqref{(1.8)}.} Similarly, we have
\begin{align}\label{1.8-1}
&\sum_{k=0}^{\tf{n}{2}-1} \df{\sin\left(\frac{ (j-1)(2k+1)\pi}{2n}\right)\sin\left(\frac{ (j+1)(2k+1)\pi}{2n}\right)}
{\sin^2\left(\frac{ (2k+1)\pi}{2n}\right)\sin^2\left(\frac{j(2k+1)\pi}{2n}\right)}
=\frac12\sum_{k=0}^{\tf{n}{2}-1} \df{\cos\left(\frac{(2k+1)\pi}{n}\right)-\cos\left(\frac{j(2k+1)\pi}{n}\right)}
{\sin^2\left(\frac{ (2k+1)\pi}{2n}\right)\sin^2\left(\frac{j(2k+1)\pi}{2n}\right)} \notag\\
&=\sum_{k=0}^{\tf{n}{2}-1} \df{\sin^2\left(\frac{j(2k+1)\pi}{2n}\right)-\sin^2\left(\frac{(2k+1)\pi}{2n}\right)}
{\sin^2\left(\frac{ (2k+1)\pi}{2n}\right)\sin^2\left(\frac{j(2k+1)\pi}{2n}\right)}
=\sum_{k=0}^{\tf{n}{2}-1}\frac{1}{\sin^2\left(\frac{(2k+1)\pi}{2n}\right)}-
\sum_{k=0}^{\tf{n}{2}-1}\frac{1}{\sin^2\left(\frac{j(2k+1)\pi}{2n}\right)}.
\end{align}
Applying \eqref{C8.6}, we obtain
\begin{align}\label{1.8-2}
&\sum_{k=0}^{\tf{n}{2}-1}\frac{1}{\sin^2\left(\frac{(2k+1)\pi}{2n}\right)}
=\sum_{k=1}^{n-1}\frac{1}{\sin^2\left(\frac{k\pi}{2n}\right)}
-\sum_{k=1}^{\tf{n}{2}-1}\frac{1}{\sin^2\left(\frac{k\pi}{n}\right)} \notag\\
&=\frac12\left(\frac{4n^2-1}{3}-1\right)-\frac12\left(\frac{n^2-1}{3}-1\right)=\frac{n^2}{2},
\end{align}
and
\begin{align}\label{1.8-3}
&\sum_{k=0}^{\tf{n}{2}-1}\frac{1}{\sin^2\left(\frac{j(2k+1)\pi}{2n}\right)}
=\sum_{k=0}^{\tf{n}{2}-1}\frac{1}{\sin^2\left(\frac{(2k+1)\pi}{n}\right)} \notag\\
&=\sum_{k=1}^{n-1}\frac{1}{\sin^2\left(\frac{k\pi}{n}\right)}-
\sum_{k=1}^{n'-1}\frac{1}{\sin^2\left(\frac{k\pi}{n'}\right)}
=\frac{n^2-1}{3}-\frac{(n')^2-1}{3}=\frac{n^2}{4}.
\end{align}
Hence, by  \eqref{1.8-1}, \eqref{1.8-2} and \eqref{1.8-3}, we finish the proof of \eqref{(1.8)}.

{\it Proof of \eqref{(1.7)}.} From \eqref{1.8-1}, it follows that
\begin{align}\label{1.7-1}
\sum_{k=0}^{\tf{n-3}{2}} \df{\sin\left(\frac{ (j-1)(2k+1)\pi}{2n}\right)\sin\left(\frac{ (j+1)(2k+1)\pi}{2n}\right)}
{\sin^2\left(\frac{(2k+1)\pi}{2n}\right)\sin^2\left(\frac{j(2k+1)\pi}{2n}\right)}
=\sum_{k=0}^{\tf{n-3}{2}}\frac{1}{\sin^2\left(\frac{(2k+1)\pi}{2n}\right)}-
\sum_{k=0}^{\tf{n-3}{2}}\frac{1}{\sin^2\left(\frac{j(2k+1)\pi}{2n}\right)}.
\end{align}
First, by \eqref{C8.6}, we have
\begin{align}\label{1.7-2}
&\sum_{k=0}^{\tf{n-3}{2}}\frac{1}{\sin^2\left(\frac{(2k+1)\pi}{2n}\right)}
=\sum_{k=1}^{n-1}\frac{1}{\sin^2\left(\frac{k\pi}{2n}\right)}
-\sum_{k=1}^{\tf{n-1}{2}}\frac{1}{\sin^2\left(\frac{k\pi}{n}\right)} \notag\\
&\qquad =\frac12\left(\frac{4n^2-1}{3}-1\right)-\frac12\cdot\frac{n^2-1}{3}=\frac{n^2-1}{2}.
\end{align}
Also,
\begin{align}\label{1.7-3}
\sum_{k=0}^{\tf{n-3}{2}}\frac{1}{\sin^2\left(\frac{j(2k+1)\pi}{2n}\right)}
=\sum_{k=0}^{\tf{n-3}{2}}\frac{1}{\sin^2\left(\frac{(2k+1)\pi}{n}\right)}
=\frac12\sum_{k=1}^{n-1}\frac{1}{\sin^2\left(\frac{k\pi}{n}\right)}=\frac{n^2-1}{6}.
\end{align}
Putting \eqref{1.7-2} and \eqref{1.7-3} into \eqref{1.7-1} yields \eqref{(1.7)}.

{\it Proof of \eqref{(1.10)}.} Analogously, we find that
\begin{align*}
&\sum_{k=0}^{\tf{n-3}{2}} \df{\sin\left(\frac{ (j-1)(2k+1)\pi}{n}\right)\sin\left(\frac{ (j+1)(2k+1)\pi}{n}\right)}
{\sin^2\left(\frac{ (2k+1)\pi}{n}\right)\sin^2\left(\frac{j(2k+1)\pi}{n}\right)}
=\sum_{k=0}^{\tf{n-3}{2}}\frac{1}{\sin^2\left(\frac{(2k+1)\pi}{n}\right)}-
\sum_{k=0}^{\tf{n-3}{2}}\frac{1}{\sin^2\left(\frac{j(2k+1)\pi}{n}\right)} \\
&\qquad =\sum_{k=0}^{\tf{n-3}{2}}\frac{1}{\sin^2\left(\frac{(2k+1)\pi}{n}\right)}-
\sum_{k=0}^{\tf{n-3}{2}}\frac{1}{\sin^2\left(\frac{(2k+1)\pi}{n}\right)}=0.
\end{align*}
which completes the proof of \eqref{(1.10)}.
\end{proof}

\section{Analogues of Ramanujan Sums}\label{ramanujan}
In his famous paper \cite{Ram1918}, \cite[p.~179]{cp}, Ramanujan began by defining the sums
\begin{equation}  \label{Ramanujansum}
c_q(n):= \sum_{\substack{k=1 \\ (k,q)=1}}^q e^{2\pi i\,kn/q},
\end{equation}
which are now known as \emph{Ramanujan Sums}, and   he derived several properties of them.  Excellent surveys on Ramanujan sums have been written by M.~Ram Murty \cite{murty} and M.~ T.~ Rassias and L\'aszl\'o T\'oth \cite{rassias}.

The evaluations in the proceeding sections, in particular, \eqref{(1.5)} and \eqref{(1.6)}, motivate us to define sums with the same summands, but now with a co-prime requirement on the indices of the summands, i.e., analogues of Ramanujan sums \eqref{Ramanujansum}.  Appearing in our proofs is the M\"{o}bius function $\mu(n)$, and, in particular, we need the property \cite[p.~111]{nz}
\begin{equation}\label{mu0}
\sum_{\substack{d|n}}\mu(d)=\begin{cases}1,\quad &\text{if } n=1,\\0,  &\text{if } n>1. \end{cases}
\end{equation}

\begin{theorem}\label{mu1}
If $k$ is an odd positive integer with $k\geq3$, then
\begin{align}\label{mu2}
\tilde{S}(k):=\sum_{\substack{1\leq \ell<k\\ \ell \textup{ odd}\\(\ell,k)=1}}(-1)^{(\ell-1)/2}\sin\left(\df{\pi\ell}{2k}\right)=\df{(-1)^{(k-1)/2}}{2}\mu(k).
\end{align}
\end{theorem}

\begin{proof} Observe that \eqref{(1.5)} can be written in the form
\begin{equation}\label{mu3}
S(k)=\sum_{\substack{1\leq \ell <k\\ \ell \textup{ odd}}}(-1)^{(\ell-1)/2}\sin\left(\df{\pi \ell}{2k}\right)=\df{ (-1)^{(k+1)/2}}{2}.
\end{equation}

Using \eqref{mu0}, we can write $\tilde{S}(k)$ as
\begin{align}\label{mu4}
\tilde{S}(k)=&\sum_{\substack{1\leq \ell<k\\ \ell \textup{ odd}}}(-1)^{(\ell-1)/2}\sin\left(\df{\pi\ell}{2k}\right)\sum_{\substack{d|\ell\\d|k}}\mu(d)
=&\sum_{\substack{d|k\\d<k}}\mu(d)\sum_{\substack{1\leq \ell<k\\ \ell \textup{ odd}\\d|\ell}}(-1)^{(\ell-1)/2}\sin\left(\df{\pi\ell}{2k}\right).
\end{align}
Now set
\begin{align}\label{mu4.5}
k= dq,\quad
\ell= db.
\end{align}
Thus, the inner sum on the right-hand side of \eqref{mu4} takes the shape
\begin{align}\label{mu5}
\sum_{\substack{1\leq \ell<k\\ \ell \textup{ odd}\\d|\ell}}(-1)^{(\ell-1)/2}\sin\left(\df{\pi\ell}{2k}\right)
=&\sum_{\substack{1\leq b<q\\ b \textup{ odd}}}(-1)^{(db-1)/2}\sin\left(\df{\pi b}{2q}\right)\notag\\
=&(-1)^{(d-1)/2}\sum_{\substack{1\leq b<q\\ b \textup{ odd}}}(-1)^{(b-1)/2}\sin\left(\df{\pi b}{2q}\right),
\end{align}
since $b$ and $d$ are both odd.  Putting \eqref{mu5} in \eqref{mu4}, recalling that $q=k/d$, and employing \eqref{mu3}, we conclude that
\begin{align}\label{mu6}
\tilde{S}(k)=&\sum_{\substack{d|k\\d<k}}(-1)^{(d-1)/2}\mu(d)\sum_{\substack{1\leq b<q\\ b \textup{ odd}}}(-1)^{(b-1)/2}\sin\left(\df{\pi b}{2q}\right)\notag\\
=&\sum_{\substack{d|k\\d<k}}(-1)^{(d-1)/2}\mu(d)\cdot\df{ (-1)^{(k/d+1)/2}}{2}\notag\\
=&\df12\sum_{\substack{d|k\\d<k}}(-1)^{(d+k/d)/2}\mu(d).
\end{align}
Now if $A$ and $B$ are odd integers, it is easily checked that $A+B\equiv AB+1\pmod 4$.  Thus,
$$ d+\df{k}{d}\equiv k+1\pmod4 \,\, \Rightarrow  \df12\left(d+\df{k}{d}\right)\equiv \df{k+1}{2}\pmod2\,\, \Rightarrow (-1)^{(d+k/d)/2}=(-1)^{(k+1)/2}.$$
Thus, by \eqref{mu6} and \eqref{mu0},
\begin{equation*}
\tilde{S}(k)=\df{(-1)^{(k+1)/2}}{2}\sum_{\substack{d|k\\d<k}}\mu(d)=\df{(-1)^{(k+1)/2}}{2}\left\{\sum_{d|k}\mu(d)-\mu(k)\right\}=\df{(-1)^{(k-1)/2}}{2}\mu(k),
\end{equation*}
which completes the proof of \eqref{mu2}.
\end{proof}

When $k$ is a prime number, it is easily seen that \eqref{(1.5)} and \eqref{mu3} are in agreement.

\begin{theorem}\label{mu7}
Let $k$ be an odd positive integer with $k\geq3$.  Then
\begin{align}\label{mu8}
\tilde{T}(k):=\sum_{\substack{1\leq \ell<k\\ \ell \textup{ odd}\\(\ell,k)=1}}(-1)^{(\ell-1)/2}\csc\left(\df{\pi\ell}{2k}\right)=\df{k}{2}
\prod_{\substack{p|k\\p\equiv 1 \pmod 4}}\left(1-\df{1}{p}\right)\prod_{\substack{p|k\\p\equiv 3 \pmod 4}}\left(1+\df{1}{p}\right),
\end{align}
where the products are taken over all primes $p\equiv 1\pmod 4$ and $p\equiv 3 \pmod 4$, respectively, with $p$ dividing $k$ in each product.
\end{theorem}

If $k$ is prime, then \eqref{mu8} yields
\begin{equation*}
\tilde{T}(k)=\begin{cases}\df{k}{2}\left(1-\df{1}{k}\right), \quad &\text{if } k\equiv 1\pmod 4,\\
\df{k}{2}\left(1+\df{1}{k}\right), \quad &\text{if } k\equiv 3\pmod 4,
\end{cases}
\end{equation*}
which is in agreement with \eqref{(1.6)}.

\begin{proof} We begin by writing \eqref{(1.6)} in the form
\begin{equation}\label{mu9}
T(k)=\sum_{\substack{1\leq \ell <k\\ \ell \textup{ odd}}}(-1)^{(\ell-1)/2}\csc\left(\df{\pi \ell}{2k}\right)=\df{k+ (-1)^{(k+1)/2}}{2}.
\end{equation}

As above, employing \eqref{mu0}, we can write $\tilde{T}(k)$ as
\begin{align}\label{mu10}
\tilde{T}(k)=&\sum_{\substack{1\leq \ell<k\\ \ell \textup{ odd}}}(-1)^{(\ell-1)/2}\csc\left(\df{\pi\ell}{2k}\right)\sum_{\substack{d|\ell\\d|k}}\mu(d)
=&\sum_{\substack{d|k\\d<k}}\mu(d)\sum_{\substack{1\leq \ell<k\\ \ell \textup{ odd}\\d|\ell}}(-1)^{(\ell-1)/2}\csc\left(\df{\pi\ell}{2k}\right).
\end{align}
Making the changes of variables \eqref{mu4.5}, we find that the inner sum above can be represented by
\begin{align}\label{mu11}
\sum_{\substack{1\leq \ell<k\\ \ell \textup{ odd}\\d|\ell}}(-1)^{(\ell-1)/2}\csc\left(\df{\pi\ell}{2k}\right)
=&\sum_{\substack{1\leq b<q\\ b \textup{ odd}}}(-1)^{(db-1)/2}\csc\left(\df{\pi b}{2q}\right)\notag\\
=&(-1)^{(d-1)/2}\sum_{\substack{1\leq b<q\\ b \textup{ odd}}}(-1)^{(b-1)/2}\csc\left(\df{\pi b}{2q}\right)\notag\\
=&\chi(d)\sum_{\substack{1\leq b<q\\ b \textup{ odd}}}(-1)^{(b-1)/2}\csc\left(\df{\pi b}{2q}\right)\notag\\
=&\chi(d)\df{q+(-1)^{(q+1)/2}}{2},
\end{align}
 where  $\chi$ denotes the non-principal character modulo 4, and we applied \eqref{mu9} for the last equality.
 Remembering that $q=k/d$, we insert \eqref{mu11} into \eqref{mu10} to arrive at
 \begin{equation}\label{mu12}
 \tilde{T}(k)=\sum_{\substack{d|k\\d<k}}\mu(d)\chi(d)\df{k/d+(-1)^{(k/d+1)/2}}{2}.
 \end{equation}
 Observe that, since $k/d$ is odd,
 \begin{equation}\label{mu12.5}
 \chi(d)(-1)^{(k/d+1)/2}=-\chi(d)\chi(k/d)=-\chi(k).
 \end{equation}
 Hence, from \eqref{mu12}, \eqref{mu12.5}, and \eqref{mu0},
 \begin{align*}
 \tilde{T}(k)=&\frac12 \sum_{\substack{d|k\\d<k}}\mu(d)\left(\chi(d)\df{k}{d}-\chi(k)\right)\notag\\
 =&\df{k}{2}\sum_{\substack{d|k\\d<k}}\df{\mu(d)\chi(d)}{d}-\df{\chi(k)}{2}\sum_{\substack{d|k\\d<k}}\mu(d)\notag\\
 =&\df{k}{2}\sum_{d|k}\df{\mu(d)\chi(d)}{d}-\df{\mu(k)\chi(k)}{2}-\df{\chi(k)}{2}\sum_{d|k}\mu(d)+\df{\mu(k)\chi(k)}{2}\notag\\
 =&\df{k}{2}\sum_{d|k}\df{\mu(d)\chi(d)}{d}\notag\\
 =&\df{k}{2}\prod_{p|k}\left(1+\df{\mu(p)\chi(p)}{p}\right)\\
 =&\df{k}{2}\prod_{p|k}\left(1-\df{\chi(p)}{p}\right)\\
 =&\df{k}{2}\prod_{\substack{p|k\\p\equiv 1 \pmod 4}}\left(1-\df{1}{p}\right)\prod_{\substack{p|k\\p\equiv 3 \pmod 4}}\left(1+\df{1}{p}\right),
  \end{align*}
  which is identical to \eqref{mu8}, and so the proof is complete.
\end{proof}

We examine further analogues of Ramanujan's sum \eqref{Ramanujansum}. Return to the four evaluations in Theorem \ref{HRJ}. We evaluate analogues in the same spirit of Theorems \ref{mu1} and \ref{mu7}.

\begin{theorem}\label{mu13}
Let $k$ be an odd positive integer with $k\geq3$, and let $j$ denote an even positive integer such that $(j,k)=1$.  Then
\begin{align}\label{mu14}
\tilde{U}_1(k,j):=&\sum_{\substack{1\leq \ell<k\\\ell \textup{ odd}\\(\ell,k)=1}}
\df{\sin\left(\df{(j-1)\ell\pi}{2k}\right)\sin\left(\df{(j+1)\ell\pi}{2k}\right)}{\sin^2\left(\df{\ell\pi}{2k}\right)\sin^2\left(\df{j\ell\pi}{2k}\right)}
=\df{k^2}{3}\prod_{p|k}\left(1-\df{1}{p^2}\right),
\end{align}
where the product is over all primes $p$ that divide $k$.
\end{theorem}

\begin{proof} Begin by writing \eqref{(1.7)} in the form
\begin{equation}\label{mu15}
U_1(n,j):=\sum_{\substack{1\leq \ell<n\\\ell \text{ odd}}} \df{\sin\left(\dfrac{ (j-1)\ell\pi}{2n}\right)\sin\left(\dfrac{ (j+1)\ell\pi}{2n}\right)}
{\sin^2\left(\dfrac{ \ell\pi}{2n}\right)\sin^2\left(\dfrac{j\ell\pi}{2n}\right)}=\df{n^2-1}{3}.
\end{equation}

From the definition \eqref{mu14}, \eqref{mu0}, \eqref{mu4.5}, and \eqref{mu15},
\begin{align}\label{mu16}
\tilde{U}_1(k,j)=&\sum_{\substack{1\leq \ell<k\\\ell \textup{ odd}}}
\df{\sin\left(\df{(j-1)\ell\pi}{2k}\right)\sin\left(\df{(j+1)\ell\pi}{2k}\right)}{\sin^2\left(\df{\ell\pi}{2k}\right)\sin^2\left(\df{j\ell\pi}{2k}\right)}
\sum_{\substack{d|\ell\\d|k}}\mu(d)\notag\\
=&\sum_{\substack{d|k\\d<k}}\mu(d)\sum_{\substack{1\leq \ell<k\\\ell \textup{ odd}\\d|\ell}}
\df{\sin\left(\df{(j-1)\ell\pi}{2k}\right)\sin\left(\df{(j+1)\ell\pi}{2k}\right)}{\sin^2\left(\df{\ell\pi}{2k}\right)\sin^2\left(\df{j\ell\pi}{2k}\right)}\notag\\
=&\sum_{\substack{d|k\\d<k}}\mu(d)\sum_{\substack{1\leq b<q\\b \textup{ odd}}}
\df{\sin\left(\df{(j-1)b\pi}{2q}\right)\sin\left(\df{(j+1)b\pi}{2q}\right)}{\sin^2\left(\df{b\pi}{2q}\right)\sin^2\left(\df{jb\pi}{2q}\right)}\notag\\
=&\sum_{\substack{d|k\\d<k}}\mu(d)U_1(q,j)\notag\\
=&\sum_{\substack{d|k\\d<k}}\mu(d)\df{q^2-1}{3}.
\end{align}
Now,
\begin{align}\label{mu17}
\sum_{\substack{d|k\\d<k}}\mu(d)\df{q^2-1}{3}
=&\sum_{d|k}\mu(d)\df{1}{3}\left(\df{k^2}{d^2}-1\right)
=\df{k^2}{3}\sum_{d|k}\df{\mu(d)}{d^2}\notag\\
=&\df{k^2}{3}\prod_{p|k}\left(1+\df{\mu(p)}{p^2}\right)
=\df{k^2}{3}\prod_{p|k}\left(1-\df{1}{p^2}\right),
\end{align}
where the product is over all primes $p$ dividing $k$.  Substituting \eqref{mu17} into \eqref{mu16}, we complete the proof of \eqref{mu14}.
\end{proof}

If $k$ is a prime, then
$$ \tilde{U}_1(k,j)=\df{k^2}{3}\left(1-\df{1}{k^2}\right)=\df{k^2-1}{3},$$
in agreement with \eqref{mu15}.

\begin{theorem}
Let $k$ be an even positive integer, and let $j\equiv 2 \pmod{4}$ with $(\frac{j}{2},\frac{k}{2})=1$.  Then
\begin{align*}
\tilde{U}_2(k,j):=&\sum_{\substack{1\leq \ell<k\\\ell \textup{ odd}\\(\ell,k)=1}}
\df{\sin\left(\df{(j-1)\ell\pi}{2k}\right)\sin\left(\df{(j+1)\ell\pi}{2k}\right)}{\sin^2\left(\df{\ell\pi}{2k}\right)\sin^2\left(\df{j\ell\pi}{2k}\right)}
=\df{k^2}{3}\prod_{p|k}\left(1-\df{1}{p^2}\right),
\end{align*}
where the product is over all primes $p$ that divide $k$.
\end{theorem}

\begin{proof}
We first rewrite \eqref{(1.8)} as
\begin{equation}\label{U2}
U_2(n,j):=\sum_{\substack{1\leq \ell<n\\\ell \text{ odd}}} \df{\sin\left(\dfrac{ (j-1)\ell\pi}{2n}\right)\sin\left(\dfrac{ (j+1)\ell\pi}{2n}\right)}
{\sin^2\left(\dfrac{ \ell\pi}{2n}\right)\sin^2\left(\dfrac{j\ell\pi}{2n}\right)}=\df{n^2}{4},
\end{equation}
where $n$ is an even positive integer.
By the change of variables \eqref{mu4.5} and \eqref{U2}, we find that
\begin{align}\label{U2-1}
\tilde{U}_2(k,j)=&\sum_{\substack{1\leq \ell<k\\\ell \textup{ odd}}}
\df{\sin\left(\df{(j-1)\ell\pi}{2k}\right)\sin\left(\df{(j+1)\ell\pi}{2k}\right)}{\sin^2\left(\df{\ell\pi}{2k}\right)\sin^2\left(\df{j\ell\pi}{2k}\right)}
\sum_{\substack{d|\ell\\d|k}}\mu(d)\notag\\
=&\sum_{\substack{d|k\\d \textup{ odd}}}\mu(d)\sum_{\substack{1\leq \ell<k\\\ell \textup{ odd}\\d|\ell}}
\df{\sin\left(\df{(j-1)\ell\pi}{2k}\right)\sin\left(\df{(j+1)\ell\pi}{2k}\right)}{\sin^2\left(\df{\ell\pi}{2k}\right)\sin^2\left(\df{j\ell\pi}{2k}\right)}\notag\\
=&\sum_{\substack{d|k\\d \textup{ odd}}}\mu(d)\sum_{\substack{1\leq b<q\\b \textup{ odd}}}
\df{\sin\left(\df{(j-1)b\pi}{2q}\right)\sin\left(\df{(j+1)b\pi}{2q}\right)}{\sin^2\left(\df{b\pi}{2q}\right)\sin^2\left(\df{jb\pi}{2q}\right)}\notag\\
=&\sum_{\substack{d|k\\d \textup{ odd}}}\mu(d)U_2(q,j)\notag\\
=&\sum_{\substack{d|k\\d \textup{ odd}}}\mu(d)\df{q^2}{4}=\frac{k^2}{4}\sum_{\substack{d|k\\d \textup{ odd}}}\df{\mu(d)}{d^2}.
\end{align}
Note that
\begin{align}\label{muodd}
\sum_{d|k}\df{\mu(d)}{d^2}=\sum_{\substack{d|k\\d \textup{ odd}}}\df{\mu(d)}{d^2}
+\sum_{\substack{d|k\\d \textup{ odd}}}\df{\mu(2d)}{4d^2}=
\sum_{\substack{d|k\\d \textup{ odd}}}\df{\mu(d)}{d^2}-\frac{1}{4}\sum_{\substack{d|k\\d \textup{ odd}}}\df{\mu(d)}{d^2}
=\frac{3}{4}\sum_{\substack{d|k\\d \textup{ odd}}}\df{\mu(d)}{d^2}.
\end{align}
Putting \eqref{muodd} into \eqref{U2-1} yields
\begin{align*}
\tilde{U}_2(k,j)=\frac{k^2}{3}\sum_{d|k}\df{\mu(d)}{d^2}=\df{k^2}{3}\prod_{p|k}\left(1-\df{1}{p^2}\right).
\end{align*}
This completes our proof.
\end{proof}

\begin{theorem}
Let $k$ be a positive integer such that $k\equiv 0 \pmod{4}$, and let $j\equiv 2 \pmod{4}$ with $(\frac{j}{2},\frac{k}{2})=1$.  Then
\begin{align*}
\tilde{U}_3(k,j):=&\sum_{\substack{1\leq \ell<k/2 \\(\ell,k)=1}}
\df{\sin\left(\df{(j-1)\ell\pi}{k}\right)\sin\left(\df{(j+1)\ell\pi}{k}\right)}{\sin^2\left(\df{\ell\pi}{k}\right)\sin^2\left(\df{j\ell\pi}{k}\right)}
=\df{k^2}{12}\prod_{p|k}\left(1-\df{1}{p^2}\right),
\end{align*}
where the product is over all primes $p$ that divide $k$.
\end{theorem}

\begin{proof}
By \eqref{(1.9)},  we have
\begin{align}\label{U3}
U_3(n,j):=&\sum_{1\leq \ell<n/2}
\df{\sin\left(\df{(j-1)\ell\pi}{n}\right)\sin\left(\df{(j+1)\ell\pi}{n}\right)}{\sin^2\left(\df{\ell\pi}{n}\right)\sin^2\left(\df{j\ell\pi}{n}\right)}
=\frac{n^2-4}{12},
\end{align}
where $n$ is an even positive integer. Now, since  $k\equiv 0 \pmod{4},$
\begin{align*}
\tilde{U}_3(k,j)&=\sum_{1\leq \ell<k/2 }
\df{\sin\left(\df{(j-1)\ell\pi}{k}\right)\sin\left(\df{(j+1)\ell\pi}{k}\right)}{\sin^2\left(\df{\ell\pi}{k}\right)\sin^2\left(\df{j\ell\pi}{k}\right)}
\sum_{\substack{d|\ell\\d|k}}\mu(d)\\
&=\sum_{\substack{d | k\\d<k/2}}\mu(d)\sum_{\substack{1\leq \ell<k/2 \\ d | \ell} }
\df{\sin\left(\df{(j-1)\ell\pi}{k}\right)\sin\left(\df{(j+1)\ell\pi}{k}\right)}{\sin^2\left(\df{\ell\pi}{k}\right)\sin^2\left(\df{j\ell\pi}{k}\right)}\\
&=\sum_{\substack{d | k\\d<k/2\\d \textup{ odd}}}\mu(d)\sum_{\substack{1\leq \ell<k/2 \\ d | \ell} }
\df{\sin\left(\df{(j-1)\ell\pi}{k}\right)\sin\left(\df{(j+1)\ell\pi}{k}\right)}{\sin^2\left(\df{\ell\pi}{k}\right)\sin^2\left(\df{j\ell\pi}{k}\right)}\\
&+\sum_{\substack{d | k\\d<k/4\\d \textup{ odd}}}\mu(2d)\sum_{\substack{1\leq \ell<k/2 \\ 2d | \ell} }
\df{\sin\left(\df{(j-1)\ell\pi}{k}\right)\sin\left(\df{(j+1)\ell\pi}{k}\right)}{\sin^2\left(\df{\ell\pi}{k}\right)\sin^2\left(\df{j\ell\pi}{k}\right)}\\
&=\sum_{\substack{d | k\\d \textup{ odd}}}\mu(d)\sum_{1\leq b<q/2}
\df{\sin\left(\df{(j-1)b\pi}{q}\right)\sin\left(\df{(j+1)b\pi}{q}\right)}{\sin^2\left(\df{b\pi}{q}\right)\sin^2\left(\df{jb\pi}{q}\right)}\\
&-\sum_{\substack{d | k\\d \textup{ odd}}}\mu(d)\sum_{1\leq b'<q'/2}
\df{\sin\left(\df{(j-1)b'\pi}{q'}\right)\sin\left(\df{(j+1)b'\pi}{q'}\right)}{\sin^2\left(\df{b'\pi}{q'}\right)\sin^2\left(\df{jb'\pi}{q'}\right)},
\end{align*}
where we used the change of variables $k=2dq'$ and $\ell=2db'$ for the last summation.
Observe that $q=k/d$ and $q'=k/(2d)$ are both even. Thus, by \eqref{U3} and \eqref{muodd},
\begin{align*}
\tilde{U}_3(k,j)&=\sum_{\substack{d | k\\d \textup{ odd}}}\mu(d)U_3(q,j)-\sum_{\substack{d | k\\d \textup{ odd}}}\mu(d)U_3(q',j)\\
&=\frac{1}{12}\sum_{\substack{d | k\\d \textup{ odd}}}\mu(d)(q^2-4)
-\frac{1}{12}\sum_{\substack{d | k\\d \textup{ odd}}}\mu(d)(q'^2-4)\\
&=\frac{1}{12}\sum_{\substack{d | k\\d \textup{ odd}}}\mu(d)\Big(\frac{k^2}{d^2}-4\Big)
-\frac{1}{12}\sum_{\substack{d | k\\d \textup{ odd}}}\mu(d)\Big(\frac{k^2}{4d^2}-4\Big)\\
&=\frac{k^2}{16}\sum_{\substack{d | k\\d \textup{ odd}}}\frac{\mu(d)}{d^2}
=\frac{k^2}{12}\sum_{d | k}\frac{\mu(d)}{d^2}=\df{k^2}{12}\prod_{p|k}\left(1-\df{1}{p^2}\right),
\end{align*}
which completes the proof.
\end{proof}

In a similar vein,  we can derive the following from \eqref{(1.10)}.
\begin{theorem}
Let $k$ be an odd positive integer with $k\geq3$, and let $j$ denote a positive integer such that $(j,k)=1$.  Then
\begin{align*}
\tilde{U}_4(k,j):=&\sum_{\substack{1\leq \ell<k\\\ell \textup{ odd}\\(\ell,k)=1}}
\df{\sin\left(\df{(j-1)\ell\pi}{k}\right)\sin\left(\df{(j+1)\ell\pi}{k}\right)}{\sin^2\left(\df{\ell\pi}{k}\right)\sin^2\left(\df{j\ell\pi}{k}\right)}
=0.
\end{align*}
\end{theorem}

\section{An Analogue of the Franel--Landau Criterion for the Riemann Hypothesis}  First define the Farey sequence and Farey fractions.
\begin{definition} Let
\begin{equation*}
\mathfrak{F}_Q=\left\{\frac{a}{q}:\,\, a,q\in \mathbb{Z}, \,\,1\leq q \leq Q,\,\, 1\leq a \leq q,\,\, (a,q)=1\right\}.
\end{equation*}
These fractions are the Farey fractions of order $Q$.
Let $N(Q)$ denote the number of Farey fractions in $\mathfrak{F}_Q$.  The Farey sequence of order $Q$ is obtained by arranging the fractions $\gamma_j$ of $\mathfrak{F}_Q$ in increasing order, i.e.,  $\gamma_1, \gamma_2, \dots, \gamma_{N(Q)}$.
\end{definition}
For brevity, set $e(x)=e^{2\pi ix}$.  Recall the definition of the Ramanujan sum in \eqref{Ramanujansum}.  From \cite[p.~10, Equation (1.5.3)]{titchmarsh}
\begin{equation}\label{titchmarsh}
c_q(1) =\sum_{\substack{k=1 \\ (k,q)=1}}^q e(k/q) =\mu(q).
\end{equation}
Summing over $q, 1\leq q\leq Q$, we find that
\begin{equation}\label{gamma,mu}
\sum_{\gamma_j\in\mathfrak{F}_Q}e(\gamma_j)=\sum_{1\leq q \leq Q}\mu(q).
\end{equation}
It is well-known \cite[p.~261]{edwards} that the Riemann Hypothesis is equivalent to the statement
\begin{equation}\label{RH}
\sum_{1\leq q \leq Q}\mu(q)=O_{\epsilon}(Q^{\frac12+\epsilon}),
\end{equation}
as $Q\to \infty$, for each $\epsilon>0$. For brevity, abbreviate the Riemann Hypothesis by RH. Hence, by \eqref{gamma,mu} and \eqref{RH}, as $Q\to\infty$,
\begin{equation}\label{RHequiv}
\text{RH holds} \iff \sum_{\gamma_j\in\mathfrak{F}_Q}e(\gamma_j)=O_{\epsilon}(Q^{\frac12+\epsilon}),
\end{equation}
for every $\epsilon>0$.

Recall that
\begin{equation*}
\sum_{q=1}^{\infty}\df{\mu(q)}{q^s} =\df{1}{\zeta(s)}, \quad \Re(s) >1.
\end{equation*}
From the product representation of $\zeta(s)$ it follows that
$$\sum_{\substack{q\geq 1\\q \textup{ odd}}}\df{\mu(q)}{q^s} =\df{1}{\left(1-\frac{1}{2^s}\right)\zeta(s)}, \quad \Re(s) >1.$$
Using an argument analogous to that yielding \eqref{RHequiv}, we can deduce that, as $Q\to \infty$,
\begin{equation}\label{RHequiv1}
\text{RH holds} \iff \sum_{\substack{1\leq q \leq Q\\q\textup{ odd}}}\mu(q)=O_{\epsilon}(Q^{\frac12+\epsilon}),
\end{equation}
 for every $\epsilon>0$.

Recall from Theorem \ref{mu1} that
\begin{equation}\label{aq}
\tilde{S}(q):=\sum_{\substack{1\leq a<q\\ a \textup{ odd}\\(a,q)=1}}(-1)^{(a-1)/2}\sin\left(\df{\pi a}{2q}\right)=\df{(-1)^{(q-1)/2}}{2}\mu(q).
\end{equation}
Let $\chi(n)$ denote the non-principal character modulo 4, i.e.,
\begin{equation*}
\chi(n)=\begin{cases}(-1)^{(n-1)/2},\quad &n \textup{ odd}\\
0,  &n \textup{ even}.
\end{cases}
\end{equation*}
Multiplying both sides of \eqref{aq} by $\chi(q)$ and
summing over odd $q$, $1\leq q\leq Q$, we find that
\begin{equation}\label{aq1}
\sum_{\substack{ a, q \textup{ odd}\\\frac{a}{q}\in\mathfrak{F}_Q}}\chi(aq)\sin\left(\df{\pi a}{2q} \right)=\df{1}{2}\sum_{\substack{1\leq q\leq Q\\q \textup{ odd}}}\mu(q).
\end{equation}
Combining \eqref{aq1} with \eqref{RHequiv1}, we conclude the following theorem.

\begin{theorem}\label{thmRH} As $Q\to \infty$,
\begin{equation}\label{RHequiv2}
\textup{RH holds} \iff \sum_{\substack{a, q \textup{ odd}\\\frac{a}{q}\in\mathfrak{F}_Q}}\chi(aq)\sin\left(\df{\pi a}{2q}\right)=O_{\epsilon}(Q^{\frac12+\epsilon}),
\end{equation}
 for each $\epsilon >0$.
\end{theorem}

The truth of \eqref{RHequiv} involves cancellations of $\sin\left(2\pi a/q\right)$ and $\cos\left(2\pi a/q\right)$.  However, in \eqref{RHequiv2}, $\sin\left(\pi a/(2q)\right)>0$.  Thus, the required cancellations in \eqref{RHequiv2} arise from the changes of sign of $\chi(aq)$.

 Theorem \ref{thmRH} gives a criterion for the Riemann Hypothesis to hold, but, as a point of contrast, the coefficients $\chi(aq)$ of the Dirichlet $L$-function, $L(s)$ := $\sum_{n=1}^{\infty}\chi(n)n^{-s}$, $\Re(s)>0$, with its own Generalized Riemann Hypothesis, appear.

 Theorem \ref{mu1} has some further
applications, which are not pursued in the present paper. Such
applications provide new connections between the Riemann Hypothesis and the
distribution of Farey fractions with parity constraints in the numerators and denominators \cite{BKZ3}.

 H.~Edwards's book \cite[pp.~263--267]{edwards} gives a detailed discussion of the Franel--Landau criterion.  With a plethora of references, the comprehensive survey \cite{cz} by C.~Cobeli and the third author provides accounts of many related problems.

\section{Sums of Quotients of Sines}\label{section2}
 As noted below, several of the results in this section were first established in \cite{HRJ} and \cite{palestine} by using Ramanujan's modular equations.  Shorter, more direct proofs were given in \cite{BKZ1}.  We now provide more systematic proofs, which are dependent upon more general identities.


\begin{theorem}\label{2-T1}
If $k$ is an odd positive integer, then
\begin{equation*}
\sum_{j=1}^{\frac{k-1}{2}}(-1)^{j-1}\frac{\sin\big(\frac{2j\pi}{k}\big)}{\sin\big(\frac{j\pi}{k}\big)}=1.
\end{equation*}
\end{theorem}

\begin{proof} Using the double angle formula for sine, we find that
\begin{align}\label{2-1}
&\sum_{j=1}^{\frac{k-1}{2}}(-1)^{j-1}\frac{\sin\big(\frac{2j\pi}{k}\big)}{\sin\big(\frac{j\pi}{k}\big)}
=2\sum_{j=1}^{\frac{k-1}{2}}(-1)^{j-1}\cos\Big(\frac{j\pi}{k}\Big)
=\sum_{j=1}^{k-1}(-1)^{j-1}\cos\Big(\frac{j\pi}{k}\Big)\notag\\
&\qquad =\Re\left\{\sum_{j=1}^{k-1}(-1)^{j-1}e^{\frac{j\pi i}{k}}\right\}
=\Re\left\{\frac{e^{\frac{\pi i}{k}}(1-e^{\frac{(k-1)\pi i}{k}})}{1+e^{\frac{\pi i}{k}}} \right\}=1.
\end{align}
\end{proof}

\begin{corollary}[Theorem 4.2 in \cite{BKZ2}; Equation (1.8) in \cite{palestine}]
\begin{align*}
&\df{\sin(6\pi/17)}{\sin(3\pi/17)}-\df{\sin(4\pi/17)}{\sin(2\pi/17)}-\df{\sin(8\pi/17)}{\sin(4\pi/17)}+\df{\sin(2\pi/17)}{\sin(\pi/17)}
\notag\\
&+\df{\sin(7\pi/17)}{\sin(5\pi/17)}-\df{\sin(5\pi/17)}{\sin(6\pi/17)}+\df{\sin(3\pi/17)}{\sin(7\pi/17)}-\df{\sin(\pi/17)}{\sin(8\pi/17)}=1.
\end{align*}
\end{corollary}

\begin{proof} Let $k=17$ in Theorem \ref{2-T1}.
\end{proof}

\begin{corollary}[Theorem 4.4 in \cite{BKZ2}; Equation (1.3) in \cite{palestine}]
\begin{align*}
&\df{\sin(4\pi/13)}{\sin(2\pi/13)}-\df{\sin(6\pi/13)}{\sin(3\pi/13)}-\df{\sin(2\pi/13)}{\sin(\pi/13)}\notag\\
&+\df{\sin(5\pi/13)}{\sin(4\pi/13)}-\df{\sin(3\pi/13)}{\sin(5\pi/13)}+\df{\sin(\pi/13)}{\sin(6\pi/13)}=-1.
\end{align*}
\end{corollary}

\begin{proof} Let $k=13$ in Theorem \ref{2-T1}.
\end{proof}


\begin{theorem}\label{2-T2}
Let $k$ be an odd positive integer and let $a$ be an integer with $1\leq a \leq \frac{k-1}{2}.$
If $k=nm$ with positive odd integers $n$ and $m,$ then
\begin{equation}\label{2-2}
\frac{\sin\big(\frac{2a\pi}{k}\big)}{\sin\big(\frac{a\pi}{k}\big)}-
\sum_{j=1}^{\frac{n-1}{2}}(-1)^{j-1}\left\{\frac{\sin\big(\frac{2(mj-a)\pi}{k}\big)}{\sin\big(\frac{(mj-a)\pi}{k}\big)}
+\frac{\sin\big(\frac{2(mj+a)\pi}{k}\big)}{\sin\big(\frac{(mj+a)\pi}{k}\big)}\right\}=0.
\end{equation}
\end{theorem}

\begin{proof}
By the double angle formula for sine and the addition formula for cosine, we can rewrite the sum in \eqref{2-2} as
\begin{align*}
&2\cos\Big(\frac{a\pi}{k}\Big)-2\sum_{j=1}^{\frac{n-1}{2}}(-1)^{j-1}
\left\{\cos\Big(\frac{(mj-a)\pi}{k}\Big)+\cos\Big(\frac{(mj+a)\pi}{k}\Big)\right\}\\
&=2\cos\Big(\frac{a\pi}{k}\Big)-4\sum_{j=1}^{\frac{n-1}{2}}(-1)^{j-1}\cos\Big(\frac{mj\pi}{k}\Big)\cos\Big(\frac{a\pi}{k}\Big)\\
&=2\cos\Big(\frac{a\pi}{k}\Big)\left\{1-2\sum_{j=1}^{\frac{n-1}{2}}(-1)^{j-1}\cos\Big(\frac{j\pi}{n}\Big)\right\}=0,
\end{align*}
where we used \eqref{2-1} for the last equality.
\end{proof}

\begin{corollary}[Theorem 4.1 in \cite{BKZ2}; Equation (1.7) in \cite{palestine}]
\begin{equation}
\df{\sin(2\pi/15)}{\sin(\pi/15)}-\df{\sin(7\pi/15)}{\sin(4\pi/15)}-\df{\sin(3\pi/15)}{\sin(6\pi/15)}=0.
\end{equation}
\end{corollary}

\begin{proof} Let $k=15,$ $a=1,$ $n=3$ and $m=5$ in Theorem \ref{2-T2}.
\end{proof}

\begin{theorem}\label{2-T3}
Let $k$ be a positive integer such that $k\equiv 1 \pmod{4}.$ Let $a_j$ and  $b_j,$ with $1\leq j\leq \frac{k-1}{4},$ be
integers such that $\{a_j\pm b_j, k-(a_j\pm b_j)\}_{j}=\{1, 2, \dots, k-1\}.$ Then,
\begin{equation}
\sum_{j=1}^{\frac{k-1}{4}}(-1)^{a_j+b_j}\frac{\sin\big(\frac{2a_j\pi}{k}\big)\sin\big(\frac{2b_j\pi}{k}\big)}
{\sin\big(\frac{a_j\pi}{k}\big)\sin\big(\frac{b_j\pi}{k}\big)}=-1.
\end{equation}
\end{theorem}

\begin{proof}  Using the double angle formula for sine and the addition formula for cosine, we deduce that
\begin{align}\label{2-3}
&\sum_{j=1}^{\frac{k-1}{4}}(-1)^{a_j+b_j}\frac{\sin\big(\frac{2a_j\pi}{k}\big)\sin\big(\frac{2b_j\pi}{k}\big)}
{\sin\big(\frac{a_j\pi}{k}\big)\sin\big(\frac{b_j\pi}{k}\big)}=
4\sum_{j=1}^{\frac{k-1}{4}}(-1)^{a_j+b_j}\cos\Big(\frac{a_j\pi}{k}\Big)\cos\Big(\frac{b_j\pi}{k}\Big)\notag\\
=&2\sum_{j=1}^{\frac{k-1}{4}}(-1)^{a_j+b_j}\left\{\cos\Big(\frac{(a_j-b_j)\pi}{k}\Big)+\cos\Big(\frac{(a_j+b_j)\pi}{k}\Big)\right\}\notag\\
=&\sum_{j=1}^{\frac{k-1}{4}}(-1)^{a_j+b_j}\left\{\cos\Big(\frac{(a_j-b_j)\pi}{k}\Big)
+\cos\Big(\frac{(a_j+b_j)\pi}{k}\Big)\right\}\notag\\
&+\sum_{j=1}^{\frac{k-1}{4}}(-1)^{k-(a_j+b_j)}\left\{\cos\Big(\frac{\big(k-(a_j-b_j)\big)\pi}{k}\Big)+\cos\Big(\frac{\big(k-(a_j+b_j)\big)\pi}{k}\Big)\right\}.
\end{align}
With the condition $\{a_j\pm b_j, k-(a_j\pm b_j)\}_{j}=\{1, 2, \dots, k-1\}$ and \eqref{2-1}, we see that the four sums on the right-hand side of \eqref{2-3} can be amalgamated into 
\begin{align*}
\sum_{j=1}^{k-1}(-1)^{j}\cos\Big(\frac{j\pi}{k}\Big)=-1.
 \end{align*}
\end{proof}

\begin{corollary}[Theorem 4.3 in \cite{BKZ2}; Equation (1.9) in \cite{palestine}]
\begin{align*}
&\df{\sin(6\pi/17)\sin(7\pi/17)}{\sin(3\pi/17)\sin(5\pi/17)}+\df{\sin(4\pi/17)\sin(\pi/17)}{\sin(2\pi/17)\sin(8\pi/17)}\notag\\
&\qquad -\df{\sin(8\pi/17)\sin(2\pi/17)}{\sin(4\pi/17)\sin(\pi/17)}-\df{\sin(5\pi/17)\sin(3\pi/17)}{\sin(6\pi/17)\sin(7\pi/17)}=-1.
\end{align*}
\end{corollary}

\begin{proof} Let $k=17,$ and set $(a_1, b_1)=(5, 3),$ $(a_2, b_2)=(8, 2),$ $(a_3, b_3)=(4, 1),$ and $(a_4, b_4)=(7, 6)$
in Theorem \ref{2-T3}.
\end{proof}

\begin{corollary}[Theorem 4.5 in \cite{BKZ2}; Equation (1.4) in \cite{palestine}]
\begin{equation*}
\df{\sin(4\pi/13)\sin(6\pi/13)}{\sin(2\pi/13)\sin(3\pi/13)}-\df{\sin(2\pi/13)\sin(3\pi/13)}{\sin(\pi/13)\sin(5\pi/13)}
-\df{\sin(5\pi/13)\sin(\pi/13)}{\sin(4\pi/13)\sin(6\pi/13)}=1.
\end{equation*}
\end{corollary}

\begin{proof}
Let $k=13,$ and set $(a_1, b_1)=(3, 2),$ $(a_2, b_2)=(5, 1),$ and $(a_3, b_3)=(6, 4)$ in Theorem \ref{2-T3}.
\end{proof}

\section{Sums of Sines from Real Quadratic Fields}\label{section5}
We consider a real quadratic field $\Q(\sqrt{d}),$ where $d$ is a positive discriminant. We denote by $h$ and $\e$ the class number and the fundamental unit of $\Q(\sqrt{d}),$ respectively. Then, the class number formula of Dirichlet can be given
by
\begin{align}\label{D}
\e^{2h}=\prod_{0<j<d}\Big[\sin\Big(\frac{j\pi}{d}\Big)\Big]^{-(\frac{j}{d})},
\end{align}
where $(\frac{j}{d})$ is the Kronecker symbol \cite[Theorem 3, p.~ 246]{bs}, \cite[Equation (1), p.~2]{ckk}. From \eqref{D}, we can easily derive the following identity.

\begin{theorem} Let $p$ be a prime with $p\equiv 1 \pmod4$ and let $\e$ be the fundamental unit of $\Q(\sqrt{p}).$ Then,
\begin{equation}
\df{\prod_{0<n<p/2}\sin\big(\frac{n\pi}{p}\big)}{\prod_{0<r<p/2}\sin\big(\frac{r\pi}{p}\big)}
\pm\df{\prod_{0<r<p/2}\sin\big(\frac{n\pi}{p}\big)}{\prod_{0<n<p/2}\sin\big(\frac{r\pi}{p}\big)}
=\e^h\pm\e^{-h},
\end{equation}
where $n$ and $r$ run over the quadratic non-residues and residues of p, respectively, and $h$ is the class number of
$\Q(\sqrt{p}).$
\end{theorem}

The fundamental unit of $\Q(\sqrt{13})$ is $(3+\sqrt{13})/2$ and the class number $h=1$.
Hence, we readily derive the following corollaries; see also \cite[Theorem 4.7]{BKZ2}.

\begin{corollary}[Theorem 4.7 in \cite{BKZ1}]
\begin{align*}
\df{\sin(6\pi/13)\sin(2\pi/13)\sin(5\pi/13)}{\sin(4\pi/13)\sin(3\pi/13)\sin(\pi/13)}-
\df{\sin(4\pi/13)\sin(3\pi/13)\sin(\pi/13)}{\sin(6\pi/13)\sin(2\pi/13)\sin(5\pi/13)}=3.
\end{align*}
\end{corollary}

\begin{corollary}
\begin{align*}
\df{\sin(6\pi/13)\sin(2\pi/13)\sin(5\pi/13)}{\sin(4\pi/13)\sin(3\pi/13)\sin(\pi/13)}+
\df{\sin(4\pi/13)\sin(3\pi/13)\sin(\pi/13)}{\sin(6\pi/13)\sin(2\pi/13)\sin(5\pi/13)}=\sqrt{13}.
\end{align*}
\end{corollary}

In a similar manner, the fundamental unit of $\Q(\sqrt{17})$ is $4+\sqrt{17}$ and the class number $h=1;$
thus we deduce the following examples.

\begin{corollary}
\begin{align*}
\df{\sin\big(\frac{3\pi}{17}\big)\sin\big(\frac{5\pi}{17}\big)\sin\big(\frac{6\pi}{17}\big)\sin\big(\frac{7\pi}{17}\big)}
{\sin\big(\frac{\pi}{17}\big)\sin\big(\frac{2\pi}{17}\big)\sin\big(\frac{4\pi}{17}\big)\sin\big(\frac{8\pi}{17}\big)}-
\df{\sin\big(\frac{\pi}{17}\big)\sin\big(\frac{2\pi}{17}\big)\sin\big(\frac{4\pi}{17}\big)\sin\big(\frac{8\pi}{17}\big)}
{\sin\big(\frac{3\pi}{17}\big)\sin\big(\frac{5\pi}{17}\big)\sin\big(\frac{6\pi}{17}\big)\sin\big(\frac{7\pi}{17}\big)}=8,
\end{align*}
and
\begin{align*}
\df{\sin\big(\frac{3\pi}{17}\big)\sin\big(\frac{5\pi}{17}\big)\sin\big(\frac{6\pi}{17}\big)\sin\big(\frac{7\pi}{17}\big)}
{\sin\big(\frac{\pi}{17}\big)\sin\big(\frac{2\pi}{17}\big)\sin\big(\frac{4\pi}{17}\big)\sin\big(\frac{8\pi}{17}\big)}+
\df{\sin\big(\frac{\pi}{17}\big)\sin\big(\frac{2\pi}{17}\big)\sin\big(\frac{4\pi}{17}\big)\sin\big(\frac{8\pi}{17}\big)}
{\sin\big(\frac{3\pi}{17}\big)\sin\big(\frac{5\pi}{17}\big)\sin\big(\frac{6\pi}{17}\big)\sin\big(\frac{7\pi}{17}\big)}=2\sqrt{17}.
\end{align*}
\end{corollary}

\section{Cotangent Sums}\label{section3}

\begin{lemma}[Lemmas 7.3 and 9.1 in \cite{BKZ1}] \label{3-L1}
If $k$ is a positive integer and $z_n=e^{2\pi inm/k}$ with $(m, k)=1,$ then
\begin{align*}
&\sum_{n=1}^{k-1}\frac{1}{z_n-1}=-\frac{k-1}{2}, \qquad  \qquad \qquad \quad   \sum_{n=1}^{k-1}\frac{z_n}{z_n-1}=\frac{k-1}{2} \\
&\sum_{n=1}^{k-1}\frac{1}{(z_n-1)^2}=-\frac{(k-1)(k-5)}{12}, \qquad
\sum_{n=1}^{k-1}\frac{z_n}{(z_n-1)^2}=-\frac{(k-1)(k+1)}{12} \\
&\sum_{n=1}^{k-1}\frac{1}{(z_n-1)^3}=\frac{(k-1)(k-3)}{8}.
\end{align*}
\end{lemma}

\begin{lemma}[Lemmas 11.4 in \cite{BKZ1}]\label{3-L2}
Let $p$ and $q$ be relatively prime positive integers  such that $p, q \geq 2.$ If
$\omega_j=e^{2\pi ij/p}$ and $\xi_j=e^{2\pi ij/q},$ then
\begin{align*}
q\sum_{j=1}^{p-1}\frac{1}{(\o_j-1)(\o_j^q-1)}+p\sum_{j=1}^{q-1}\frac{1}{(\xi_j-1)(\xi_j^p-1)}
=-\frac{1}{12}(p^2+q^2-9pq+3p+3q+1).
\end{align*}
\end{lemma}

\begin{theorem}\label{3-T1}
Let $p, q$ and $\mu$ be positive integers with $(p, q)=(q, \mu)=(p, \mu)=1.$
If $\o_n=e^{2\pi in/p}$, $\xi_n=e^{2\pi in/q}$ and $\nu_n=e^{2\pi in/\mu},$ then
\begin{align}\label{3-1}
&p\mu\sum_{n=1}^{q-1}\frac{1}{(\xi_n^p-1)(\xi_n^{\mu}-1)}+q\mu\sum_{j=1}^{p-1}\frac{1}{(\o_j^q-1)(\o_j^{\mu}-1)}
+pq\sum_{i=1}^{\mu-1}\frac{1}{(\nu_i^p-1)(\nu_i^{q}-1)} \notag\\
&\qquad =pq\mu-\frac{1}{12}(p^2+q^2+\mu^2)-\frac{1}{4}(pq+q\mu+p\mu).
\end{align}
\end{theorem}

Theorem \ref{3-T1} is a beautiful three sum  relation, analogous to the reciprocity theorem in Lemma \ref{3-L2}.

\begin{proof} By taking logarithmic derivatives of $x^k-1=\prod_{j=1}^{k}(x-z_j)$ where $z_j=e^{2\pi ij/k}$, we can easily derive
\begin{align}\label{xk1}
\frac{k}{x^k-1}=\sum_{j=1}^k\frac{z_j}{x-z_j}=\frac{1}{x-1}+\sum_{j=1}^{k-1}\frac{z_j}{x-z_j},
\end{align}
and
\begin{align}\label{xk2}
\sum_{j=1}^{k-1}\frac{1}{x-z_j}=\frac{kx^{k-1}}{x^k-1}-\frac{1}{x-1}.
\end{align}
Using \eqref{xk1}, we find that
\begin{align}\label{I}
&p\mu\sum_{n=1}^{q-1}\frac{1}{(\xi_n^p-1)(\xi_n^{\mu}-1)}=
\sum_{n=1}^{q-1}\Big(\frac{1}{\xi_n-1}+\sum_{j=1}^{p-1}\frac{\o_j}{\xi_n-\o_j}\Big)\Big(\frac{1}{\xi_n-1}+\sum_{i=1}^{\mu-1}\frac{\nu_i}{\xi_n-\nu_i}\Big) \notag\\
&=\sum_{n=1}^{q-1}\left\{\frac{1}{(\xi_n-1)^2}+\frac{1}{\xi_n-1}\sum_{j=1}^{p-1}\frac{\o_j}{\xi_n-\o_j}
+\frac{1}{\xi_n-1}\sum_{i=1}^{\mu-1}\frac{\nu_i}{\xi_n-\nu_i}+
\sum_{j=1}^{p-1}\frac{\o_j}{\xi_n-\o_j}\sum_{i=1}^{\mu-1}\frac{\nu_i}{\xi_n-\nu_i}\right\}\notag\\
&=I_1+I_2+I_3+I_4.
\end{align}
By Lemma \ref{3-L1},
\begin{equation}\label{I1}
I_1:=\sum_{n=1}^{q-1}\frac{1}{(\xi_n-1)^2}=-\frac{(q-1)(q-5)}{12}.
\end{equation}
Also, using  \eqref{xk2} and  Lemma \ref{3-L1}, we obtain
\begin{align}\label{I2}
I_2&:=\sum_{n=1}^{q-1}\frac{1}{\xi_n-1}\sum_{j=1}^{p-1}\frac{\o_j}{\xi_n-\o_j}=
\sum_{j=1}^{p-1}\frac{\o_j}{1-\o_j}\sum_{n=1}^{q-1}\Big(\frac{1}{\xi_n-1}-\frac{1}{\xi_n-\o_j} \Big) \notag\\
&=-\sum_{j=1}^{p-1}\frac{\o_j}{\o_j-1}\sum_{n=1}^{q-1}\frac{1}{\xi_n-1}-
\sum_{j=1}^{p-1}\frac{\o_j}{\o_j-1}\sum_{n=1}^{q-1}\frac{1}{\o_j-\xi_n}  \notag\\
&=\frac{(p-1)(q-1)}{4}-\sum_{j=1}^{p-1}\frac{\o_j}{\o_j-1}\left\{\frac{q\o_j^{q-1}}{\o_j^q-1}-\frac{1}{\o_j-1}\right\} \notag\\
&=\frac{(p-1)(q-1)}{4}-\sum_{j=1}^{p-1}\frac{1}{\o_j-1}\left\{q+\frac{q}{\o_j^q-1}-\frac{\o_j}{\o_j-1}\right\}\notag\\
&=\frac{(p-1)(q-1)}{4}+\frac{q(p-1)}{2}-\sum_{j=1}^{p-1}\frac{q}{(\o_j-1)(\o_j^q-1)}-\frac{(p-1)(p+1)}{12}\notag\\
&=-\frac{p^2-9pq+9q+3p-4}{12}-q\sum_{j=1}^{p-1}\frac{1}{(\o_j-1)(\o_j^q-1)}.
\end{align}
Similarly, we can easily see that
\begin{align}\label{I3}
I_3&:=\sum_{n=1}^{q-1}\frac{1}{\xi_n-1}\sum_{i=1}^{\mu-1}\frac{\nu_i}{\xi_n-\nu_i}
=-\frac{\mu^2-9q\mu+9q+3\mu-4}{12}-q\sum_{i=1}^{\mu-1}\frac{1}{(\nu_i-1)(\nu_i^q-1)}.
\end{align}
Next, utilizing \eqref{xk2}, we have
\begin{align}\label{I4}
I_4&:=\sum_{n=1}^{q-1}\sum_{j=1}^{p-1}\frac{\o_j}{\xi_n-\o_j}\sum_{i=1}^{\mu-1}\frac{\nu_i}{\xi_n-\nu_i}
=\sum_{j=1}^{p-1}\sum_{i=1}^{\mu-1}\frac{\o_j\nu_i}{\o_j-\nu_i}
\sum_{n=1}^{q-1}\left\{\frac{1}{\xi_n-\o_j}-\frac{1}{\xi_n-\nu_i}\right\} \notag\\
&=\sum_{j=1}^{p-1}\sum_{i=1}^{\mu-1}\frac{\o_j\nu_i}{\o_j-\nu_i}
\left\{\frac{1}{\o_j-1}-\frac{q\o_j^{q-1}}{\o_j^q-1}-\frac{1}{\nu_i-1}+\frac{q\nu_i^{q-1}}{\nu_i^q-1}\right\} \notag\\
&=\sum_{j=1}^{p-1}\left\{\frac{\o_j}{\o_j-1}-\frac{q\o_j^{q}}{\o_j^q-1}\right\}\sum_{i=1}^{\mu-1}\frac{\nu_i}{\o_j-\nu_i}
+\sum_{i=1}^{\mu-1}\left\{\frac{\nu_i}{\nu_i-1}-\frac{q\nu_i^{q}}{\nu_i^q-1}\right\}\sum_{j=1}^{p-1}\frac{\o_j}{\nu_i-\o_j}\notag\\
&=I_{4,1}+I_{4,2}.
\end{align}
Observe that, by \eqref{xk2},
\begin{align}\label{I4-1}
I_{4,1}&:=\sum_{j=1}^{p-1}\left\{\frac{\o_j}{\o_j-1}-\frac{q\o_j^{q}}{\o_j^q-1}\right\}\sum_{i=1}^{\mu-1}\frac{\nu_i}{\o_j-\nu_i} \notag\\
&=\sum_{j=1}^{p-1}\left\{\frac{\o_j}{\o_j-1}-q-\frac{q}{\o_j^q-1}\right\}
\left\{-(\mu-1)+\o_j\sum_{i=1}^{\mu-1}\frac{1}{\o_j-\nu_i} \right\}\notag\\
&=\sum_{j=1}^{p-1}\left\{\frac{\o_j}{\o_j-1}-q-\frac{q}{\o_j^q-1}\right\}
\left\{-(\mu-1)+\frac{\mu\o_j^{\mu}}{\o_j^{\mu}-1}-\frac{\o_j}{\o_j-1}\right\} \notag\\
&=\sum_{j=1}^{p-1}\left\{1-q+\frac{1}{\o_j-1}-\frac{q}{\o_j^q-1}\right\}
\left\{\frac{\mu}{\o_j^{\mu}-1}-\frac{1}{\o_j-1}\right\} \notag\\
&=(1-q)\mu\sum_{j=1}^{p-1}\frac{1}{\o_j^{\mu}-1}-(1-q)\sum_{j=1}^{p-1}\frac{1}{\o_j-1}-\sum_{j=1}^{p-1}\frac{1}{(\o_j-1)^2} \notag\\
&\quad +\mu\sum_{j=1}^{p-1}\frac{1}{(\o_j-1)(\o_j^{\mu}-1)}-q\mu\sum_{j=1}^{p-1}\frac{1}{(\o_j^q-1)(\o_j^{\mu}-1)}
+q\sum_{j=1}^{p-1}\frac{1}{(\o_j-1)(\o_j^{q}-1)}\notag\\
&=\frac{(p-1)(q-1)(\mu-1)}{2}+\frac{(p-1)(p-5)}{12}+\mu\sum_{j=1}^{p-1}\frac{1}{(\o_j-1)(\o_j^{\mu}-1)}\notag\\
&\quad -q\mu\sum_{j=1}^{p-1}\frac{1}{(\o_j^q-1)(\o_j^{\mu}-1)}+q\sum_{j=1}^{p-1}\frac{1}{(\o_j-1)(\o_j^{q}-1)},
\end{align}
where we applied Lemma \ref{3-L1} for the last equality. Analogously, we obtain
\begin{align}\label{I4-2}
I_{4,2}&:=\sum_{i=1}^{\mu-1}\left\{\frac{\nu_i}{\nu_i-1}-\frac{q\nu_i^{q}}{\nu_i^q-1}\right\}
\sum_{j=1}^{p-1}\frac{\o_j}{\nu_i-\o_j}\notag\\
&=\frac{(p-1)(q-1)(\mu-1)}{2}+\frac{(\mu-1)(\mu-5)}{12}+p\sum_{i=1}^{\mu-1}\frac{1}{(\nu_i-1)(\nu_i^{p}-1)}\notag\\
&\quad -pq\sum_{i=1}^{\mu-1}\frac{1}{(\nu_i^q-1)(\nu_i^{p}-1)}+q\sum_{i=1}^{\mu-1}\frac{1}{(\nu_i-1)(\nu_i^{q}-1)}.
\end{align}
Putting \eqref{I1}--\eqref{I4-2} into \eqref{I} yields
\begin{align}\label{I'}
&p\mu\sum_{n=1}^{q-1}\frac{1}{(\xi_n^p-1)(\xi_n^{\mu}-1)}=-\frac{(q-1)(q-5)}{12}-\frac{p^2-9pq+9q+3p-4}{12}\notag\\
&\quad -\frac{\mu^2-9q\mu+9q+3\mu-4}{12}+(p-1)(q-1)(\mu-1)+\frac{(p-1)(p-5)}{12}\notag\\
&\quad +\frac{(\mu-1)(\mu-5)}{12}+\mu\sum_{j=1}^{p-1}\frac{1}{(\o_j-1)(\o_j^{\mu}-1)}
+p\sum_{i=1}^{\mu-1}\frac{1}{(\nu_i-1)(\nu_i^{p}-1)}\notag\\
&\quad-q\mu\sum_{j=1}^{p-1}\frac{1}{(\o_j^q-1)(\o_j^{\mu}-1)}
-pq\sum_{i=1}^{\mu-1}\frac{1}{(\nu_i^q-1)(\nu_i^{p}-1)}.
\end{align}
By Lemma \ref{3-L2}, we have
\begin{align*}
\mu\sum_{j=1}^{p-1}\frac{1}{(\o_j-1)(\o_j^{\mu}-1)}+p\sum_{i=1}^{\mu-1}\frac{1}{(\nu_i-1)(\nu_i^{p}-1)}
=-\frac{1}{12}(p^2+{\mu}^2-9p\mu+3p+3\mu+1).
\end{align*}
Putting this  into \eqref{I'} and simplifying, we arrive at
\begin{align*}
&p\mu\sum_{n=1}^{q-1}\frac{1}{(\xi_n^p-1)(\xi_n^{\mu}-1)}=pq\mu-\frac{1}{12}(p^2+q^2+\mu^2)-\frac{1}{4}(pq+q\mu+p\mu)\\
&\qquad-q\mu\sum_{j=1}^{p-1}\frac{1}{(\o_j^q-1)(\o_j^{\mu}-1)}
-pq\sum_{i=1}^{\mu-1}\frac{1}{(\nu_i^q-1)(\nu_i^{p}-1)}.
\end{align*}
This completes our proof.
\end{proof}

\begin{theorem}\label{3-T2}
If $p, q$ and $\mu$ are positive integers such that $(p, q)=(q, \mu)=(p, \mu)=1,$ and $\xi_n=e^{2\pi in/q},$ then
\begin{align*}
\sum_{n=1}^{q-1}\cot\Big(\frac{\pi np}{q}\Big)\cot\Big(\frac{\pi n\mu}{q}\Big)
=q-1-4\sum_{n=1}^{q-1}\frac{1}{(\xi_n^p-1)(\xi_n^{\mu}-1)}.
\end{align*}
\end{theorem}

\begin{proof} Observe that
\begin{align*}
&\sum_{n=1}^{q-1}\cot\Big(\frac{\pi np}{q}\Big)\cot\Big(\frac{\pi n\mu}{q}\Big)=
-\sum_{n=1}^{q-1}\frac{\xi_n^p+1}{\xi_n^p-1}\frac{\xi_n^{\mu}+1}{\xi_n^{\mu}-1}
=-\sum_{n=1}^{q-1}\Big(1+\frac{2}{\xi_n^p-1}\Big)\Big(1+\frac{2}{\xi_n^{\mu}-1}\Big)\\
&=-\sum_{n=1}^{q-1}\left\{1+\frac{2}{\xi_n^p-1}+\frac{2}{\xi_n^{\mu}-1}+\frac{4}{(\xi_n^p-1)(\xi_n^{\mu}-1)}\right\}\\
&=q-1-4\sum_{n=1}^{q-1}\frac{1}{(\xi_n^p-1)(\xi_n^{\mu}-1)},
\end{align*}
where we employed Lemma \ref{3-L1} for the last equality.
\end{proof}

In order to define Dedekind sums, first set
\begin{equation*}
((x))=
\begin{cases}
x-[x]-\frac12, & \text{if } ~ x\not\in\mathbb{Z}, \\
0, & \text{otherwise}.
\end{cases}
\end{equation*}

\begin{definition} If $h$ and $k$ are relatively prime positive integers, then the Dedekind sum $s(h,k)$ is defined by
\begin{equation*}
s(h,k):=\sum_{n=1}^{k-1}\Big(\Big(\frac{n}{k}\Big)\Big)\Big(\Big(\frac{hn}{k}\Big)\Big).
\end{equation*}
\end{definition}
The Dedekind sums satisfy a famous reciprocity theorem \cite[p.~4]{rg}.  If $h$ and $k$ are coprime positive integers, then
\begin{equation*}
s(h,k)+s(k,h)=-\df14+\df{1}{12}\left(\df{h}{k}+\df{1}{hk}+\df{k}{h}\right).
\end{equation*}

Moreover, $s(h,k)$ has the representation \cite[p.~26]{rg}
\begin{equation*}
s(h,k)=\df{1}{4k}\sum_{n=1}^{k-1}\cot\Big(\dfrac{\pi n}{k}\Big)\cot\Big(\dfrac{\pi hn}{k}\Big).
\end{equation*}

From Theorems \ref{3-T1} and \ref{3-T2}, we can easily deduce the following corollary,
which is equivalent to the three sum relation for Dedekind sums, and which was proved by non-elementary means in \cite[Equation (3.10), p.~292]{Berndt}.

\begin{corollary}
If $p, q$ and $\mu$ are positive integers such that $(p, q)=(q, \mu)=(p, \mu)=1,$ then
\begin{align*}
&p\mu\sum_{n=1}^{q-1}\cot\Big(\frac{\pi np}{q}\Big)\cot\Big(\frac{\pi n\mu}{q}\Big)
+q\mu\sum_{n=1}^{p-1}\cot\Big(\frac{\pi nq}{p}\Big)\cot\Big(\frac{\pi n\mu}{p}\Big)\notag\\
&\qquad +pq\sum_{n=1}^{\mu-1}\cot\Big(\frac{\pi np}{\mu}\Big)\cot\Big(\frac{\pi nq}{\mu}\Big)
=\frac{p^2+q^2+\mu^2}{3}-pq\mu.
\end{align*}
\end{corollary}

\begin{corollary}[Theorem 2.14 in \cite{BY}]
Let $p$ and $q$ be coprime positive integers, and let $p+q=\mu c,$ where $\mu$ and $c$ are positive integers.
Then,
\begin{align*}
&p\sum_{n=1}^{q-1}\cot\Big(\frac{\pi np}{q}\Big)\cot\Big(\frac{\pi n\mu}{q}\Big)
+q\sum_{n=1}^{p-1}\cot\Big(\frac{\pi nq}{p}\Big)\cot\Big(\frac{\pi n\mu}{p}\Big)\\
&\quad=\frac{1}{3\mu}(p^2+q^2+\mu^2)+\frac{pq}{3\mu}(\mu^2-6\mu+2).
\end{align*}
\end{corollary}

\begin{proof}
By Theorem \ref{3-T2}, we find that
\begin{align}\label{3-C2}
&p\sum_{n=1}^{q-1}\cot\Big(\frac{\pi np}{q}\Big)\cot\Big(\frac{\pi n\mu}{q}\Big)
+q\sum_{n=1}^{p-1}\cot\Big(\frac{\pi nq}{p}\Big)\cot\Big(\frac{\pi n\mu}{p}\Big) \notag\\
&=2pq-p-q-4p\sum_{n=1}^{q-1}\frac{1}{(\xi_n^p-1)(\xi_n^{\mu}-1)}-4q\sum_{n=1}^{p-1}\frac{1}{(\o_n^q-1)(\o_n^{\mu}-1)}.
\end{align}
Now, if $\nu_n=e^{2\pi in/\mu},$ then we see that
\begin{align}\label{3-C3}
\sum_{n=1}^{\mu-1}\frac{1}{(\nu_n^q-1)(\nu_n^{p}-1)}=\sum_{n=1}^{\mu-1}\frac{1}{(\nu_n^{-p}-1)(\nu_n^{p}-1)}
=-\sum_{n=1}^{\mu-1}\frac{\nu_n^p}{(\nu_n^{p}-1)^2}=\frac{(\mu-1)(\mu+1)}{12},
\end{align}
where we used the condition $q=\mu c-p$ and Lemma \ref{3-L1}.

Therefore, from \eqref{3-C2}, \eqref{3-C3}, and Theorem \ref{3-T1}, it follows that
\begin{align*}
&p\sum_{n=1}^{q-1}\cot\Big(\frac{\pi np}{q}\Big)\cot\Big(\frac{\pi n\mu}{q}\Big)
+q\sum_{n=1}^{p-1}\cot\Big(\frac{\pi nq}{p}\Big)\cot\Big(\frac{\pi n\mu}{p}\Big) \notag\\
&=2pq-p-q-4\left\{pq-\frac{1}{12\mu}(p^2+q^2+\mu^2)-\frac{1}{4\mu}(pq+q\mu+p\mu)-\frac{pq}{12\mu}(\mu-1)(\mu+1)\right\}\\
&=\frac{1}{3\mu}(p^2+q^2+\mu^2)+\frac{pq}{3\mu}(\mu^2-6\mu+2).
\end{align*}
Hence, we complete the proof.
\end{proof}

The modified Dedekind sum $S_{\a, \b}(h, k)$ \cite{Berndt0, Berndt} is defined by
\begin{align*}
S_{\a, \b}(h, k):=\sum_{j \pmod{hk}}e^{2(j\a/h+j\b/k)\pi i}((j/hk))((jh'/k)),
\end{align*}
where $\a, \b, h$ and $k$ are positive integers with $(h, k)=1$ and $hh'\equiv 1 \pmod k.$

\begin{lemma}[Lemma 6.1 in \cite{Berndt}]
For $k>1,$
\begin{align}
S_{\a, \b}(h, k)=\frac{1}{4k}\sum_{j=1}^{k-1}\cot\left(\Big(\frac{j}{k}+\frac{\a}{h}+\frac{\b}{k}\Big)\pi\right)
\cot\left(\frac{jh}{k}\pi\right).\label{cotcot}
\end{align}
\end{lemma}

We generalize \eqref{cotcot} by introducing another parameter. Define for each positive integer $\mu,$
\begin{align*}
S_{\a, \b}(h, k | \mu)&:=\frac{1}{4k}\sum_{j=1}^{k-1}\cot\left(\mu\Big(\frac{j}{k}+\frac{\a}{h}+\frac{\b}{k}\Big)\pi\right)
\cot\left(\frac{jh}{k}\pi\right).
\end{align*}

\begin{theorem}\label{3-T6}
Let $h, k$ and $\mu$ be pairwise relatively prime positive integers with $h, k \geq 2,$ 
and let $\xi_n=e^{2\pi in/k}$ and $\o_n=e^{2\pi in/h}.$ For positive integers $\a$ and $\b$ such that  $(\a, h)=(\b, k)=1,$
\begin{align*}
4kS_{\a, \b}(h, k | \mu)=\frac{2}{\xi_{\b}^{\mu}\o_{\a}^{\mu}-1}-\frac{2k}{\o^{k\mu}_{\a}-1}
-4\sum_{j=1}^{k-1}\frac{1}{(\xi^{\mu}_{j+\b}\o^{\mu}_{\a}-1)(\xi_j^h-1)}.
\end{align*}
\end{theorem}

\begin{proof}
Using an argument similar to that in the proof of Theorem \ref{3-T2}, we obtain
\begin{align}\label{3-T61}
4kS_{\a, \b}(h, k | \mu)&=-\sum_{j=1}^{k-1}\frac{(\xi_{j+\b}^{\mu}\o_{\a}^{\mu}+1)}{(\xi_{j+\b}^{\mu}\o_{\a}^{\mu}-1)}
\frac{(\xi_j^{h}+1)}{(\xi_j^{h}-1)} \notag\\
&=-\sum_{j=1}^{k-1}\Big(1+\frac{2}{\xi_{j+\b}^{\mu}\o_{\a}^{\mu}-1}\Big)\Big(1+\frac{2}{\xi_j^{h}-1}\Big)\notag\\
&=-\sum_{j=1}^{k-1}\left\{1+\frac{2}{\xi_{j+\b}^{\mu}\o_{\a}^{\mu}-1}
+\frac{2}{\xi_j^{h}-1}+\frac{4}{(\xi_{j+\b}^{\mu}\o_{\a}^{\mu}-1)(\xi_j^h-1)}\right\}\notag\\
&=-\sum_{j=1}^{k-1}\left\{\frac{2}{\xi_{j+\b}^{\mu}\o_{\a}^{\mu}-1}+\frac{4}{(\xi_{j+\b}^{\mu}\o_{\a}^{\mu}-1)(\xi_j^h-1)}\right\},
\end{align}
where we applied Lemma \ref{3-L1} for the last equality.

Taking logarithmic derivatives of $x^m-a^m=\prod_{j=1}^{m}(x-z_ja),$ where  $z_j=e^{2\pi ij/m},$ we deduce that
\begin{align}\label{xam}
\sum_{j=1}^m\frac{1}{x-az_j}=\frac{mx^{m-1}}{x^m-a^m},
\end{align}
and also
\begin{align}\label{xam2}
\sum_{j=1}^m\frac{1}{(x-az_j)^2}=-\frac{m(m-1)x^{m-2}}{x^m-a^m}+\frac{m^2x^{2m-2}}{(x^m-a^m)^2}.
\end{align}
Thus, if $a^m\neq 1,$
\begin{align}\label{a1}
\sum_{j=1}^m\frac{1}{az_j-1}=\frac{m}{a^m-1},
\end{align}
and
\begin{align}\label{a2}
\sum_{j=1}^m\frac{1}{(az_j-1)^2}=\frac{m(m-1)}{a^m-1}+\frac{m^2}{(a^m-1)^2}.
\end{align}
Applying \eqref{a1} to  \eqref{3-T61} yields
\begin{align}\label{T}
4kS_{\a, \b}(h, k | \mu)&=-2\left\{\frac{k}{\o_{\a}^{\mu k}-1}-\frac{1}{\xi_{\b}^{\mu}\o_{\a}^{\mu}-1}\right\}
-4\sum_{j=1}^{k-1}\frac{1}{(\xi^{\mu}_{j+\b}\o^{\mu}_{\a}-1)(\xi_j^h-1)},
\end{align}
which completes the proof.
\end{proof}

\begin{theorem}\label{3-T7}
Let $k,  h$ and $\mu$ be pairwise relatively prime positive integers with $k, h \geq 2,$ 
and let $\xi_n=e^{2\pi in/k},$ $\o_n=e^{2\pi in/h}$ and $\nu_n=e^{2\pi in/\mu}.$
For positive integers $\a$ and $\b$ such that  $(\a, h)=(\b, k)=1,$
\begin{align*}
&\sum_{j=1}^{\mu}\cot\left(k\Big(\frac{j}{\mu}+\frac{\a}{h}\Big)\pi\right)
\cot\left(h\Big(\frac{j}{\mu}+\frac{\b}{k}\Big)\pi\right)\\
&\quad =-\mu\left(1+\frac{2}{\o_{\a}^{k\mu}-1}+\frac{2}{\xi_{\b}^{h\mu}-1}\right)-
4\sum_{j=1}^{\mu}\frac{1}{(\nu_{j}^{k}\o_{\a}^{k}-1)(\nu_j^h\xi_{\b}^{h}-1)}.
\end{align*}
\end{theorem}

\begin{proof}
Similarly to the proof of Theorem \ref{3-T2}, we have
\begin{align*}
&\sum_{j=1}^{\mu}\cot\left(k\Big(\frac{j}{\mu}+\frac{\a}{h}\Big)\pi\right)
\cot\left(h\Big(\frac{j}{k}+\frac{\b}{k}\Big)\pi\right)=
-\sum_{j=1}^{\mu}\frac{(\nu_{j}^{k}\o_{\a}^{k}+1)}{(\nu_{j}^{k}\o_{\a}^{k}-1)}
\frac{(\nu_j^h\xi_{\b}^{h}+1)}{(\nu_j^h\xi_{\b}^{h}-1)}\\
&=-\sum_{j=1}^{\mu}\left(1+\frac{2}{\nu_{j}^{k}\o_{\a}^{k}-1}\right)\left(1+\frac{2}{\nu_j^h\xi_{\b}^{h}-1}\right)\\
&=-\sum_{j=1}^{\mu}\left\{1+\frac{2}{\nu_{j}^{k}\o_{\a}^{k}-1}+\frac{2}{\nu_j^h\xi_{\b}^{h}-1}+
\frac{4}{(\nu_{j}^{k}\o_{\a}^{k}-1)(\nu_j^h\xi_{\b}^{h}-1)} \right\}\\
&=-\mu\left(1+\frac{2}{\o_{\a}^{k\mu}-1}+\frac{2}{\xi_{\b}^{h\mu}-1}\right)-
4\sum_{j=1}^{\mu}\frac{1}{(\nu_{j}^{k}\o_{\a}^{k}-1)(\nu_j^h\xi_{\b}^{h}-1)},
\end{align*}
where we used \eqref{a1} for the last equality.
\end{proof}

\begin{theorem}[Three Sum Relation]\label{3-T8}
Let $h, k$ and $\mu$ be pairwise relatively prime positive integers with $h, k \geq 2,$ 
and let $\xi_n=e^{2\pi in/k},$ $\o_n=e^{2\pi in/h}$ and $\nu_n=e^{2\pi in/\mu}.$
For positive integers $\a$ and $\b$ such that  $(\a, h)=(\b, k)=1,$
\begin{align*}
&h\mu\sum_{j=1}^{k-1}\frac{1}{(\xi^{\mu}_{j+\b}\o^{\mu}_{\a}-1)(\xi_j^h-1)}
+k\mu\sum_{j=1}^{h-1}\frac{1}{(\xi^{\mu}_{\b}\o^{\mu}_{\a+j}-1)(\o_j^k-1)}\\
&\qquad +hk\sum_{j=1}^{\mu}\frac{1}{(\o_{\a}^{k}\nu_{j}^{k}-1)(\xi_{\b}^{h}\nu_j^h-1)}\\
&=-hk\mu\Big(\frac{1}{\o_{\a}^{k\mu}-1}+\frac{1}{\xi_{\b}^{h\mu}-1}\Big)
+\mu\Big(\mu+\frac{k}{2}+\frac{h}{2}\Big)\frac{1}{\xi_{\b}^{\mu}\o_{\a}^{\mu}-1}
+\frac{\mu^2}{(\xi_{\b}^{\mu}\o_{\a}^{\mu}-1)^2}.
\end{align*}
\end{theorem}

\begin{proof}
By \eqref{a1} and Lemma \ref{3-L1}, we find that
\begin{align} \label{3-T81}
&h\mu\sum_{j=1}^{k-1}\frac{1}{(\xi^{\mu}_{j+\b}\o^{\mu}_{\a}-1)(\xi_j^h-1)}
=\sum_{j=1}^{k-1}\sum_{\ell=1}^{\mu}\frac{1}{\xi_{j+\b}\o_{\a}\nu_{\ell}-1}\sum_{t=1}^{h}\frac{1}{\xi_j\o_t-1} \notag\\
&=-\sum_{j=1}^{k-1}\sum_{\ell=1}^{\mu}\sum_{t=1}^{h}\frac{1}{\xi_{\b}\o_{\a}\nu_{\ell}-\o_t}
\left\{\frac{\xi_{\b}\o_{\a}\nu_{\ell}}{\xi_{j+\b}\o_{\a}\nu_{\ell}-1}-\frac{\o_t}{\xi_j\o_t-1}\right\}\notag\\
&=-\sum_{\ell=1}^{\mu}\sum_{t=1}^{h}\frac{\xi_{\b}\o_{\a}\nu_{\ell}}{\xi_{\b}\o_{\a}\nu_{\ell}-\o_t}
\sum_{j=1}^{k-1}\frac{1}{\xi_{j+\b}\o_{\a}\nu_{\ell}-1}
+\sum_{\ell=1}^{\mu}\sum_{t=1}^{h}\frac{\o_{t}}{\xi_{\b}\o_{\a}\nu_{\ell}-\o_t}
\sum_{j=1}^{k-1}\frac{1}{\xi_j\o_t-1}\notag\\
&=-\sum_{\ell=1}^{\mu}\left\{h+\sum_{t=1}^{h}\frac{1}{\xi_{\b}\o_{\a-t}\nu_{\ell}-1} \right\}
\left\{\frac{k}{\o_{\a}^k\nu_{\ell}^k-1}-\frac{1}{\xi_{\b}\o_{\a}\nu_{\ell}-1}\right\}\notag\\
&\quad+\sum_{\ell=1}^{\mu}\sum_{t=1}^{h-1}\frac{1}{\xi_{\b}\o_{\a-t}\nu_{\ell}-1}
\Big(\frac{k}{\o_t^k-1}-\frac{1}{\o_t-1}\Big)
+\sum_{\ell=1}^{\mu}\frac{1}{\xi_{\b}\o_{\a}\nu_{\ell}-1}\sum_{j=1}^{k-1}\frac{1}{\xi_j-1} \notag\\
&=-\sum_{\ell=1}^{\mu}\frac{hk}{\o_{\a}^k\nu_{\ell}^k-1}+\sum_{\ell=1}^{\mu}\frac{h}{\xi_{\b}\o_{\a}\nu_{\ell}-1}
-\sum_{\ell=1}^{\mu}\frac{hk}{(\xi_{\b}^h\nu_{\ell}^h-1)(\o_{\a}^k\nu_{\ell}^k-1)} \notag\\
&\quad +\sum_{\ell=1}^{\mu}\frac{h}{(\xi_{\b}^h\nu_{\ell}^h-1)(\xi_{\b}\o_{\a}\nu_{\ell}-1)}
+\sum_{t=1}^{h-1}\Big(\frac{k}{\o_t^k-1}-\frac{1}{\o_t-1}\Big)\sum_{\ell=1}^{\mu}\frac{1}{\xi_{\b}\o_{\a-t}\nu_{\ell}-1}\notag\\
&\qquad \qquad -\frac{\mu}{(\xi_{\b}^{\mu}\o_{\a}^{\mu}-1)}\frac{(k-1)}{2}\notag\\
&=-\frac{hk\mu}{\o_{\a}^{k\mu}-1}+\frac{h\mu}{\xi_{\b}^{\mu}\o_{\a}^{\mu}-1}
-\sum_{\ell=1}^{\mu}\frac{hk}{(\xi_{\b}^h\nu_{\ell}^h-1)(\o_{\a}^k\nu_{\ell}^k-1)}
+\sum_{\ell=1}^{\mu}\frac{h}{(\xi_{\b}^h\nu_{\ell}^h-1)(\xi_{\b}\o_{\a}\nu_{\ell}-1)}\notag\\
&\quad +\sum_{t=1}^{h-1}\Big(\frac{k}{\o_t^k-1}-\frac{1}{\o_t-1}\Big)\frac{\mu}{(\xi_{\b}^{\mu}\o_{\a-t}^{\mu}-1)}
-\frac{\mu(k-1)}{2(\xi_{\b}^{\mu}\o_{\a}^{\mu}-1)}.
\end{align}
Observe that the last summation in \eqref{3-T81} is
\begin{align}\label{3-T82}
&\sum_{t=1}^{h-1}\Big(\frac{k}{\o_t^k-1}-\frac{1}{\o_t-1}\Big)\frac{\mu}{(\xi_{\b}^{\mu}\o_{\a-t}^{\mu}-1)}
=\sum_{t=1}^{h-1}\Big(\frac{k}{\o_{-t}^k-1}-\frac{1}{\o_{-t}-1}\Big)\frac{\mu}{(\xi_{\b}^{\mu}\o_{\a+t}^{\mu}-1)}\notag\\
&=-\sum_{t=1}^{h-1}\Big(\frac{k\o_{t}^k}{\o_{t}^k-1}-\frac{\o_{t}}{\o_{t}-1}\Big)\frac{\mu}{(\xi_{\b}^{\mu}\o_{\a+t}^{\mu}-1)}\notag\\
&=-\sum_{t=1}^{h-1}\Big(k-1+\frac{k}{\o_{t}^k-1}-\frac{1}{\o_{t}-1}\Big)\frac{\mu}{(\xi_{\b}^{\mu}\o_{\a+t}^{\mu}-1)}\notag\\
&=-(k-1)\mu\sum_{t=1}^{h-1}\frac{1}{(\xi_{\b}^{\mu}\o_{\a+t}^{\mu}-1)}
-k\mu\sum_{t=1}^{h-1}\frac{1}{(\xi_{\b}^{\mu}\o_{\a+t}^{\mu}-1)(\o_{t}^k-1)}\notag\\
&\qquad +\mu\sum_{t=1}^{h-1}\frac{1}{(\xi_{\b}^{\mu}\o_{\a+t}^{\mu}-1)(\o_{t}-1)}\notag\\
&=-(k-1)\mu\Big(\frac{h}{\xi_{\b}^{h\mu}-1}-\frac{1}{\xi_{\b}^{\mu}\o_{\a}^{\mu}-1}\Big)
-k\mu\sum_{t=1}^{h-1}\frac{1}{(\xi_{\b}^{\mu}\o_{\a+t}^{\mu}-1)(\o_{t}^k-1)}\notag\\
&\qquad +\mu\sum_{t=1}^{h-1}\frac{1}{(\xi_{\b}^{\mu}\o_{\a+t}^{\mu}-1)(\o_{t}-1)}.
\end{align}
Also,
\begin{align}\label{3-T83}
&\mu\sum_{t=1}^{h-1}\frac{1}{(\xi_{\b}^{\mu}\o_{\a+t}^{\mu}-1)(\o_{t}-1)}
=\sum_{t=1}^{h-1}\frac{1}{\o_{t}-1}\sum_{\ell=1}^{\mu}\frac{1}{\xi_{\b}\o_{\a+t}\nu_{\ell}-1}\notag\\
&=\sum_{\ell=1}^{\mu}\sum_{t=1}^{h-1}\left\{\frac{1}{\o_{t}-1}
-\frac{\xi_{\b}\o_{\a}\nu_{\ell}}{\xi_{\b}\o_{\a+t}\nu_{\ell}-1}\right\}\frac{1}{\xi_{\b}\o_{\a}\nu_{\ell}-1} \notag\\
&=-\frac{(h-1)\mu}{2(\xi_{\b}^{\mu}\o_{\a}^{\mu}-1)}
-\sum_{\ell=1}^{\mu}\left(1+\frac{1}{\xi_{\b}\o_{\a}\nu_{\ell}-1}\right)\sum_{t=1}^{h-1}\frac{1}{\xi_{\b}\o_{\a+t}\nu_{\ell}-1}\notag\\
&=-\frac{(h-1)\mu}{2(\xi_{\b}^{\mu}\o_{\a}^{\mu}-1)}
-\sum_{\ell=1}^{\mu}\left(1+\frac{1}{\xi_{\b}\o_{\a}\nu_{\ell}-1}\right)
\left(\frac{h}{\xi_{\b}^h\nu_{\ell}^h-1}-\frac{1}{\xi_{\b}\o_{\a}\nu_{\ell}-1}\right)\notag\\
&=-\frac{(h-1)\mu}{2(\xi_{\b}^{\mu}\o_{\a}^{\mu}-1)}-\frac{h\mu}{\xi_{\b}^{h\mu}-1}
+\frac{\mu}{\xi_{\b}^{\mu}\o_{\a}^{\mu}-1}-\sum_{\ell=1}^{\mu}\frac{h}{(\xi_{\b}^h\nu_{\ell}^h-1)(\xi_{\b}\o_{\a}\nu_{\ell}-1)}\notag\\
&\qquad +\sum_{\ell=1}^{\mu}\frac{1}{(\xi_{\b}\o_{\a}\nu_{\ell}-1)^2}\notag\\
&=-\frac{(h-1)\mu}{2(\xi_{\b}^{\mu}\o_{\a}^{\mu}-1)}-\frac{h\mu}{\xi_{\b}^{h\mu}-1}
+\frac{\mu}{\xi_{\b}^{\mu}\o_{\a}^{\mu}-1}-\sum_{\ell=1}^{\mu}\frac{h}{(\xi_{\b}^h\nu_{\ell}^h-1)(\xi_{\b}\o_{\a}\nu_{\ell}-1)}\notag\\
&\qquad+\frac{\mu(\mu-1)}{\xi_{\b}^{\mu}\o_{\a}^{\mu}-1}+\frac{\mu^2}{(\xi_{\b}^{\mu}\o_{\a}^{\mu}-1)^2},
\end{align}
where we used \eqref{a2} for the last equality.

Putting \eqref{3-T83} into \eqref{3-T82} yields
\begin{align}\label{3-T84}
&\sum_{t=1}^{h-1}\Big(\frac{k}{\o_t^k-1}-\frac{1}{\o_t-1}\Big)\frac{\mu}{(\xi_{\b}^{\mu}\o_{\a-t}^{\mu}-1)} \notag\\
&=-\frac{kh\mu}{\xi_{\b}^{h\mu}-1}+\frac{\mu(2k-h+2\mu-1)}{2(\xi_{\b}^{\mu}\o_{\a}^{\mu}-1)}
-k\mu\sum_{t=1}^{h-1}\frac{1}{(\xi_{\b}^{\mu}\o_{\a+t}^{\mu}-1)(\o_{t}^k-1)}\notag\\
&\qquad -\sum_{\ell=1}^{\mu}\frac{h}{(\xi_{\b}^h\nu_{\ell}^h-1)(\xi_{\b}\o_{\a}\nu_{\ell}-1)}
+\frac{\mu^2}{(\xi_{\b}^{\mu}\o_{\a}^{\mu}-1)^2}.
\end{align}
Lastly, using \eqref{3-T84} in \eqref{3-T81}, we arrive at
\begin{align*}
&h\mu\sum_{j=1}^{k-1}\frac{1}{(\xi^{\mu}_{j+\b}\o^{\mu}_{\a}-1)(\xi_j^h-1)}
=-\frac{hk\mu}{\o_{\a}^{k\mu}-1}-\frac{hk\mu}{\xi_{\b}^{h\mu}-1}+\frac{\mu(k+h+2\mu)}{2(\xi_{\b}^{\mu}\o_{\a}^{\mu}-1)}
+\frac{\mu^2}{(\xi_{\b}^{\mu}\o_{\a}^{\mu}-1)^2}\notag\\
&\qquad -k\mu\sum_{t=1}^{h-1}\frac{1}{(\xi_{\b}^{\mu}\o_{\a+t}^{\mu}-1)(\o_{t}^k-1)}
-hk\sum_{\ell=1}^{\mu}\frac{1}{(\xi_{\b}^h\nu_{\ell}^h-1)(\o_{\a}^k\nu_{\ell}^k-1)},
\end{align*}
which completes our proof.
\end{proof}

\begin{theorem}[Reciprocity Theorem]\label{3-T9}
Let $h, k$ and $\mu$ be pairwise relatively prime positive integers with $h, k \geq 2,$ 
and let $\xi_n=e^{2\pi in/k},$ $\o_n=e^{2\pi in/h}$ and $\nu_n=e^{2\pi in/\mu}.$
For positive integers $\a$ and $\b$ such that  $(\a, h)=(\b, k)=1,$
\begin{align*}
4hk\mu \Big(S_{\a, \b}(h, k | \mu)+S_{\b, \a}(k, h | \mu)\Big)
=&\mu^2\csc^2\left(\mu\Big(\frac{\a}{h}+\frac{\b}{k}\Big)\pi\right)-hk\mu\notag\\
&-hk\sum_{j=1}^{\mu}\cot\left(k\Big(\frac{j}{\mu}+\frac{\a}{h}\Big)\pi\right)
\cot\left(h\Big(\frac{j}{\mu}+\frac{\b}{k}\Big)\pi\right).
\end{align*}
\end{theorem}

\begin{proof}
By Theorems \ref{3-T6}, \ref{3-T7} and \ref{3-T8}, we can easily derive
\begin{align*}
&4hk\mu\Big(S_{\a, \b}(h, k | \mu)+S_{\b, \a}(k, h | \mu)\Big)=-\frac{4\mu^2}{\xi_{\b}^{\mu}\o_{\a}^{\mu}-1}
-\frac{4\mu^2}{(\xi_{\b}^{\mu}\o_{\a}^{\mu}-1)^2} \notag\\
&\qquad+\frac{2hk\mu}{\o_{\a}^{k\mu}-1}+\frac{2hk\mu}{\xi_{\b}^{h\mu}-1}
+4hk\sum_{j=1}^{\mu}\frac{1}{(\o_{\a}^{k}\nu_{j}^{k}-1)(\xi_{\b}^{h}\nu_j^h-1)}\notag\\
&=-4\mu^2\frac{\xi_{\b}^{\mu}\o_{\a}^{\mu}}{(\xi_{\b}^{\mu}\o_{\a}^{\mu}-1)^2}
-hk\mu-hk\sum_{j=1}^{\mu}\cot\left(k\Big(\frac{j}{\mu}+\frac{\a}{h}\Big)\pi\right)
\cot\left(h\Big(\frac{j}{\mu}+\frac{\b}{k}\Big)\pi\right)\notag\\
&=\mu^2\csc^2\left(\mu\Big(\frac{\a}{h}+\frac{\b}{k}\Big)\pi\right)-hk\mu
-hk\sum_{j=1}^{\mu}\cot\left(k\Big(\frac{j}{\mu}+\frac{\a}{h}\Big)\pi\right)
\cot\left(h\Big(\frac{j}{\mu}+\frac{\b}{k}\Big)\pi\right).
\end{align*}
\end{proof}

\begin{corollary}[Theorem 6.2 in \cite{Berndt}]
Let $\a, \b, h$ and $k$ be positive integers with $(h, k)=(\a, h)=(\b, k)=1$ and $h, k \geq 2.$
Then,
\begin{align*}
4hk\Big(S_{\a, \b}(h, k)+S_{\b, \a}(k, h)\Big)=\csc^2\left(\pi\Big(\frac{\a}{h}+\frac{\b}{k}\Big)\right)
-hk\left(1+\cot\Big(\frac{k\a\pi}{h}\Big)\cot\Big(\frac{h\b\pi}{k}\Big)\right).
\end{align*}
\end{corollary}

\begin{proof}
Set $\mu=1$ in Theorem \ref{3-T9}.
\end{proof}

Next, we provide further analogous reciprocity theorems and three sum relations.
Define for a positive integer $\mu$
\begin{align*}
T_{\a, \b}(h, k \mid \mu):=\frac{1}{4k}\sum_{j=1}^{k-1}\tan\left(\mu\Big(\frac{j}{k}+\frac{\a}{h}+\frac{\b}{k}\Big)\pi\right)
\cot\left(\frac{jh}{k}\pi\right).
\end{align*}

\begin{theorem}\label{3-T10}
Let $h, k$ and $\mu$ be  pairwise relatively prime odd positive integers such that $h,k \geq 3,$
and let $\xi_n=e^{2\pi in/k}$ and $\o_n=e^{2\pi in/h}.$
If $\a$ and $\b$ are either $0$ or positive integers such that $(\a, h)=1$ and $(\b, k)=1,$ then
\begin{align*}
4kT_{\a, \b}(h, k | \mu)=\frac{2}{\xi_{\b}^{\mu}\o_{\a}^{\mu}+1}-\frac{2k}{\o^{k\mu}_{\a}+1}
-4\sum_{j=1}^{k-1}\frac{1}{(\xi^{\mu}_{j+\b}\o^{\mu}_{\a}+1)(\xi_j^h-1)}.
\end{align*}
\end{theorem}

\begin{proof}
Observe that
\begin{align}
4kT_{\a, \b}(h, k | \mu)&=\sum_{j=1}^{k-1}\frac{(\xi_{j+\b}^{\mu}\o_{\a}^{\mu}-1)}{(\xi_{j+\b}^{\mu}\o_{\a}^{\mu}+1)}
\frac{(\xi_j^{h}+1)}{(\xi_j^{h}-1)} \notag\\
&=\sum_{j=1}^{k-1}\Big(1-\frac{2}{\xi_{j+\b}^{\mu}\o_{\a}^{\mu}+1}\Big)\Big(1+\frac{2}{\xi_j^{h}-1}\Big)\notag\\
&=\sum_{j=1}^{k-1}\left\{1-\frac{2}{\xi_{j+\b}^{\mu}\o_{\a}^{\mu}+1}
+\frac{2}{\xi_j^{h}-1}-\frac{4}{(\xi_{j+\b}^{\mu}\o_{\a}^{\mu}+1)(\xi_j^h-1)}\right\}\notag\\
&=-\sum_{j=1}^{k-1}\left\{\frac{2}{\xi_{j+\b}^{\mu}\o_{\a}^{\mu}+1}+\frac{4}{(\xi_{j+\b}^{\mu}\o_{\a}^{\mu}+1)(\xi_j^h-1)}\right\},\label{star}
\end{align}
where we applied Lemma \ref{3-L1} for the last equality.

From \eqref{xam} and \eqref{xam2}, it follows that  for $a^m\neq -1$ and $z_j=e^{2\pi ij/m}$ with $m$ odd,
\begin{align}\label{a+1}
\sum_{j=1}^m\frac{1}{az_j+1}=\frac{m}{a^m+1},
\end{align}
and
\begin{align}\label{a+12}
\sum_{j=1}^m\frac{1}{(az_j+1)^2}=-\frac{m(m-1)}{a^m+1}+\frac{m^2}{(a^m+1)^2}.
\end{align}
Thus, by \eqref{a+1},
\begin{align}
\sum_{j=1}^{k-1}\frac{2}{\xi_{j+\b}^{\mu}\o_{\a}^{\mu}+1}=
\frac{2k}{\o^{k\mu}_{\a}+1}-\frac{2}{\xi_{\b}^{\mu}\o_{\a}^{\mu}+1}.\label{starstar}
\end{align}
Substituting \eqref{starstar} into \eqref{star}, we complete the proof.
\end{proof}

\begin{theorem}\label{3-T11}
Let $k,  h$ and $\mu$ be pairwise relatively prime odd positive integers with $k, h \geq 3,$ 
and let $\xi_n=e^{2\pi in/k},$ $\o_n=e^{2\pi in/h}$ and $\nu_n=e^{2\pi in/\mu}.$
If $\a$ and $\b$ are either $0$ or positive integers such that $(\a, h)=1$ and $(\b, k)=1,$
\begin{align*}
&\sum_{j=1}^{\mu}\tan\left(k\Big(\frac{j}{\mu}+\frac{\a}{h}\Big)\pi\right)
\tan\left(h\Big(\frac{j}{\mu}+\frac{\b}{k}\Big)\pi\right)\\
&\quad =\mu\left(\frac{2}{\o_{\a}^{k\mu}+1}+\frac{2}{\xi_{\b}^{h\mu}+1}-1\right)-
4\sum_{j=1}^{\mu}\frac{1}{(\nu_{j}^{k}\o_{\a}^{k}+1)(\nu_j^h\xi_{\b}^{h}+1)}.
\end{align*}
\end{theorem}

\begin{proof}
Analogously, we have
\begin{align*}
&\sum_{j=1}^{\mu}\tan\left(k\Big(\frac{j}{\mu}+\frac{\a}{h}\Big)\pi\right)
\tan\left(h\Big(\frac{j}{k}+\frac{\b}{k}\Big)\pi\right)=
-\sum_{j=1}^{\mu}\frac{(\nu_{j}^{k}\o_{\a}^{k}-1)}{(\nu_{j}^{k}\o_{\a}^{k}+1)}
\frac{(\nu_j^h\xi_{\b}^{h}-1)}{(\nu_j^h\xi_{\b}^{h}+1)}\\
&=-\sum_{j=1}^{\mu}\left(1-\frac{2}{\nu_{j}^{k}\o_{\a}^{k}+1}\right)\left(1-\frac{2}{\nu_j^h\xi_{\b}^{h}+1}\right)\\
&=-\sum_{j=1}^{\mu}\left\{1-\frac{2}{\nu_{j}^{k}\o_{\a}^{k}+1}-\frac{2}{\nu_j^h\xi_{\b}^{h}+1}+
\frac{4}{(\nu_{j}^{k}\o_{\a}^{k}+1)(\nu_j^h\xi_{\b}^{h}+1)} \right\}\\
&=-\mu\left(1-\frac{2}{\o_{\a}^{k\mu}+1}-\frac{2}{\xi_{\b}^{h\mu}+1}\right)-
4\sum_{j=1}^{\mu}\frac{1}{(\nu_{j}^{k}\o_{\a}^{k}+1)(\nu_j^h\xi_{\b}^{h}+1)},
\end{align*}
where we employed \eqref{a+1} for the last equality.
\end{proof}

Repeating similar  arguments used in the proof of Theorem \ref{3-T8}, we can derive the following theorem.

\begin{theorem}[Three Sum Relation]\label{3-T12}
Let $h, k$ and $\mu$ be pairwise relatively prime odd positive integers with $h, k \geq 3,$ 
and let $\xi_n=e^{2\pi in/k},$ $\o_n=e^{2\pi in/h}$ and $\nu_n=e^{2\pi in/\mu}.$
If $\a$ and $\b$ are either $0$ or positive integers such that $(\a, h)=1$ and $(\b, k)=1,$
\begin{align*}
&h\mu\sum_{j=1}^{k-1}\frac{1}{(\xi^{\mu}_{j+\b}\o^{\mu}_{\a}+1)(\xi_j^h-1)}
+k\mu\sum_{j=1}^{h-1}\frac{1}{(\xi^{\mu}_{\b}\o^{\mu}_{\a+j}+1)(\o_j^k-1)}\\
&\qquad -hk\sum_{j=1}^{\mu}\frac{1}{(\o_{\a}^{k}\nu_{j}^{k}+1)(\xi_{\b}^{h}\nu_j^h+1)}\\
&=-hk\mu\Big(\frac{1}{\o_{\a}^{k\mu}+1}+\frac{1}{\xi_{\b}^{h\mu}+1}\Big)
+\mu\Big(\mu+\frac{k}{2}+\frac{h}{2}\Big)\frac{1}{\xi_{\b}^{\mu}\o_{\a}^{\mu}+1}
-\frac{\mu^2}{(\xi_{\b}^{\mu}\o_{\a}^{\mu}+1)^2}.
\end{align*}
\end{theorem}


\begin{theorem}[Reciprocity Theorem]\label{3-T13}
Let $h, k$ and $\mu$ be pairwise relatively prime odd positive integers with $h, k \geq 3,$
and let $\xi_n=e^{2\pi in/k},$ $\o_n=e^{2\pi in/h}$ and $\nu_n=e^{2\pi in/\mu}.$
If $\a$ and $\b$ are either $0$ or positive integers such that $(\a, h)=1$ and $(\b, k)=1,$ then
\begin{align*}
4hk\mu \Big(T_{\a, \b}(h, k | \mu)+T_{\b, \a}(k, h | \mu)\Big)
=&-\mu^2\sec^2\left(\mu\Big(\frac{\a}{h}+\frac{\b}{k}\Big)\pi\right)+hk\mu\notag\\
&+hk\sum_{j=1}^{\mu}\tan\left(k\Big(\frac{j}{\mu}+\frac{\a}{h}\Big)\pi\right)
\tan\left(h\Big(\frac{j}{\mu}+\frac{\b}{k}\Big)\pi\right).
\end{align*}
\end{theorem}

\begin{proof}
By Theorems \ref{3-T10}, \ref{3-T11} and \ref{3-T12},
\begin{align*}
&4hk\mu\Big(T_{\a, \b}(h, k | \mu)+T_{\b, \a}(k, h | \mu)\Big)=-\frac{4\mu^2}{\xi_{\b}^{\mu}\o_{\a}^{\mu}+1}
+\frac{4\mu^2}{(\xi_{\b}^{\mu}\o_{\a}^{\mu}+1)^2} \notag\\
&\qquad+\frac{2hk\mu}{\o_{\a}^{k\mu}+1}+\frac{2hk\mu}{\xi_{\b}^{h\mu}+1}
-4hk\sum_{j=1}^{\mu}\frac{1}{(\o_{\a}^{k}\nu_{j}^{k}+1)(\xi_{\b}^{h}\nu_j^h+1)}\notag\\
&=-4\mu^2\frac{\xi_{\b}^{\mu}\o_{\a}^{\mu}}{(\xi_{\b}^{\mu}\o_{\a}^{\mu}+1)^2}
+hk\mu+hk\sum_{j=1}^{\mu}\tan\left(k\Big(\frac{j}{\mu}+\frac{\a}{h}\Big)\pi\right)
\tan\left(h\Big(\frac{j}{\mu}+\frac{\b}{k}\Big)\pi\right)\notag\\
&=-\mu^2\sec^2\left(\mu\Big(\frac{\a}{h}+\frac{\b}{k}\Big)\pi\right)+hk\mu
+hk\sum_{j=1}^{\mu}\tan\left(k\Big(\frac{j}{\mu}+\frac{\a}{h}\Big)\pi\right)
\tan\left(h\Big(\frac{j}{\mu}+\frac{\b}{k}\Big)\pi\right).
\end{align*}
\end{proof}

\begin{corollary}[Three Sum Relation]
Let $h, k$ and $\mu$ be pairwise relatively prime odd positive integers with $h, k \geq 3.$ Then,
\begin{align*}
&h\mu\sum_{j=1}^{k-1}\tan\Big(\dfrac{\mu j\pi }{k}\Big)\cot\Big(\dfrac{h j\pi }{k}\Big)
+k\mu\sum_{j=1}^{h-1}\tan\Big(\dfrac{\mu j\pi}{h}\Big)\cot\Big(\dfrac{kj\pi}{h}\Big)\\
&\quad -hk\sum_{j=1}^{\mu}\tan\Big(\dfrac{kj\pi }{\mu}\Big)\tan\Big(\dfrac{hj\pi}{\mu}\Big)=-\mu^2+hk\mu.
\end{align*}
\end{corollary}

 \begin{proof}
 Set $\a=\b=0$ in Theorem \ref{3-T13}.
 \end{proof}

\begin{corollary}[Three Sum Relation]
Let $h, k$ and $\mu$ be pairwise relatively prime odd positive integers with $h, k \geq 3.$
If $\a$ is either $0$ or a positive integer such that $(a, h)=1,$ then
\begin{align*}
&h\mu\sum_{j=1}^{k-1}\tan\left(\mu\Big(\frac{j}{k}+\frac{\a}{h}\Big)\pi\right)\cot\Big(\dfrac{h j\pi }{k}\Big)
+k\mu\sum_{j=1}^{h-1}\tan\left(\mu\Big(\frac{j}{h}+\frac{\a}{h}\Big)\pi\right)\cot\Big(\dfrac{kj\pi}{h}\Big)\\
&\quad -hk\sum_{j=1}^{\mu}\tan\left(k\Big(\frac{j}{\mu}+\frac{\a}{h}\Big)\pi\right)\tan\Big(\dfrac{hj\pi}{\mu}\Big)
=-\mu^2\sec^2\Big(\frac{\mu\a}{h}\pi\Big)+hk\mu.
\end{align*}
\end{corollary}

 \begin{proof}
 Set $\b=0$ in Theorem \ref{3-T13}.
 \end{proof}

 In a similar vein, we can also derive the following identities.

\begin{theorem}
Let $h, k$ and $\mu$ be  pairwise relatively prime odd positive integers where $h, k \geq 3,$
and let $\xi_n=e^{2\pi in/k}$ and $\o_n=e^{2\pi in/h}$. 
If $\a$ is a positive integer such that $(\a, h)=1,$ and $\b$ is  either $0$ or a positive integer with $(\b, k)=1,$
\begin{align*}
&\sum_{j=1}^{k-1}\cot\left(\mu\Big(\frac{j}{k}+\frac{\a}{h}+\frac{\b}{k}\Big)\pi\right)
\tan\left(\frac{jh}{k}\pi\right)\\
&\qquad =\frac{2k}{\o^{k\mu}_{\a}-1}-\frac{2}{\xi_{\b}^{\mu}\o_{\a}^{\mu}-1}
-4\sum_{j=1}^{k-1}\frac{1}{(\xi^{\mu}_{j+\b}\o^{\mu}_{\a}-1)(\xi_j^h+1)},
\end{align*}
\begin{align*}
&\sum_{j=1}^{h-1}\tan\left(\mu\Big(\frac{j}{h}+\frac{\a}{h}+\frac{\b}{k}\Big)\pi\right)
\tan\left(\frac{jk}{h}\pi\right)\\
&\qquad =\frac{2h}{\xi_{\b}^{h\mu}+1}-\frac{2}{\xi_{\b}^{\mu}\o_{\a}^{\mu}+1}
-4\sum_{j=1}^{h-1}\frac{1}{(\xi^{\mu}_{j}\o^{\mu}_{\a+j}+1)(\o_j^k+1)},
\end{align*}
and
\begin{align*}
&\sum_{j=1}^{\mu}\cot\left(k\Big(\frac{j}{\mu}+\frac{\a}{h}\Big)\pi\right)
\tan\left(h\Big(\frac{j}{\mu}+\frac{\b}{k}\Big)\pi\right)\\
&\quad =\mu\left(1+\frac{2}{\o_{\a}^{k\mu}-1}-\frac{2}{\xi_{\b}^{h\mu}+1}\right)-
4\sum_{j=1}^{\mu}\frac{1}{(\nu_{j}^{k}\o_{\a}^{k}-1)(\nu_j^h\xi_{\b}^{h}+1)}.
\end{align*}
\end{theorem}

 \begin{theorem}
Let $h, k$ and $\mu$ be pairwise relatively prime odd positive integers with $h, k \geq 3,$ 
and let $\xi_n=e^{2\pi in/k},$ $\o_n=e^{2\pi in/h}$ and $\nu_n=e^{2\pi in/\mu}.$
If $\a$ is a positive integer such that $(\a, h)=1,$ and $\b$ is  either $0$ or a positive integer with $(\b, k)=1,$
\begin{align*}
&h\mu\sum_{j=1}^{k-1}\frac{1}{(\xi^{\mu}_{j+\b}\o^{\mu}_{\a}-1)(\xi_j^h+1)}
-k\mu\sum_{j=1}^{h-1}\frac{1}{(\xi^{\mu}_{\b}\o^{\mu}_{\a+j}+1)(\o_j^k+1)}\\
&\qquad +hk\sum_{j=1}^{\mu}\frac{1}{(\o_{\a}^{k}\nu_{j}^{k}-1)(\xi_{\b}^{h}\nu_j^h+1)}\\
&=hk\mu\Big(\frac{1}{\o_{\a}^{k\mu}-1}-\frac{1}{\xi_{\b}^{h\mu}+1}\Big)
-\frac{h\mu}{2(\xi_{\b}^{\mu}\o_{\a}^{\mu}-1)}
+\frac{\mu(3k-2)}{2(\xi_{\b}^{\mu}\o_{\a}^{\mu}+1)}.
\end{align*}
\end{theorem}

\begin{theorem}[Three Sum Relation]
Let $h, k$ and $\mu$ be  pairwise relatively prime odd positive integers where $h, k \geq 3,$
and let $\xi_n=e^{2\pi in/k}$ and $\o_n=e^{2\pi in/h}.$
If $\a$ is a positive integer such that $(\a, h)=1,$ and $\b$ is  either $0$ or a positive integer with $(\b, k)=1,$ then
\begin{align*}
&h\mu\sum_{j=1}^{k-1}\cot\left(\mu\Big(\frac{j}{k}+\frac{\a}{h}+\frac{\b}{k}\Big)\pi\right)\tan\left(\frac{jh}{k}\pi\right)
-k\mu\sum_{j=1}^{h-1}\tan\left(\mu\Big(\frac{j}{h}+\frac{\a}{h}+\frac{\b}{k}\Big)\pi\right)\tan\left(\frac{jk}{h}\pi\right)\\
&\qquad +hk\sum_{j=1}^{\mu}\cot\left(k\Big(\frac{j}{\mu}+\frac{\a}{h}\Big)\pi\right)
\tan\left(h\Big(\frac{j}{\mu}+\frac{\b}{k}\Big)\pi\right)
=hk\mu-\frac{4\mu(k-1)}{\xi_{\b}^{\mu}\o_{\a}^{\mu}+1}.
\end{align*}
\end{theorem}


\begin{thebibliography}{00}
\bibitem{Berndt0}
B.~C.~Berndt, \emph{Generalized Eisenstein series and modified Dedekind sums}, J.~ Reine Angew.~ Math.~ \textbf{272} (1975), 182--193.

\bibitem{Berndt}
B.~C.~Berndt, \emph{Reciprocity theorems for Dedekind sums and generalizations}, Adv.~ Math.~ \textbf{23} (1977), 285--316.

\bibitem{BY}
B.~C.~Berndt and B.~P.~Yeap, \emph{Explicit evaluations and reciprocity theorems for finite trigonometric sums}, Adv.~Appl.~Math.~\textbf{29} (2002), 358--385.

\bibitem{BKZ1}
B.~C.~Berndt, S.~Kim, and A.~Zaharescu, \emph{Exact evaluations and reciprocity theorems for finite trigonometric sums}, Research in the Mathematical Sciences, ~\textbf{10}:40 (2023) (70 pages).

\bibitem{BKZ2}
B.~C.~Berndt, S.~Kim, and A.~Zaharescu, \emph{Finite trigonometric sums arising from Ramanujan's theta functions},
Ramanujan J., ~\textbf{63} (2024), 673--685.

\bibitem{BKZ3}
B.~C.~Berndt, S.~Kim, and A.~Zaharescu, paper in preparation.

\bibitem{bz}
B.~C.~Berndt and A.~Zaharescu, \emph{Finite trigonometric sums and class numbers},
Math.~Ann.~{\bf 330} (2004), 551--575.

\bibitem{bs}
A.~I.~Borevich and I.~R.~Shafarevich, \emph{Number Theory},
Pure Appl.~ Math., Vol. ~20,
Academic Press, New York--London, 1966.

\bibitem{ckk}
K.~Chakraborty, S.~Kanemitsu, and T.~Kuzumaki, \emph{On the class number formula of certain real quadratic fields}, Hardy--Ramanujan J.~\textbf{36} (2013), 1--7.

\bibitem{cz}
C.~Cobeli and A.~Zaharescu, \emph{The Haros--Farey sequence at two hundred years; A survey}, Acta Univ.~ Apulensis Math.~Inform.~(2003), no.~5, 1--38.

\bibitem{edwards}
H.~M.~Edwards, \emph{Riemann's Zeta Function}, Academic Press, New York, 1974.

\bibitem{franel}
J.~Franel, \emph{Les suites de Farey et les probl\`{e}mes des nombres premiers}, Nachr.~Akad.~Wiss.~G\"{o}ttingen Math.-Phys.~Kl.~1924, 198--201.

\bibitem{HRJ}
K.~N.~Harshitha, K.~R.~Vasuki, and M.~V.~Yathirajsharma, \emph{Trigonometric sums through Ramanujan’s theory of theta functions},
Ramanujan J. \textbf{57} (2022), 931--948. 

\bibitem{landau}
E.~Landau, \emph{Bemerkungen zu vorstehenden Abhandlung von Herrn Franel}, Nachr.~Akad.~Wiss.~G\"{o}ttingen Math.-Phys.~Kl.~1924, 202--206; \emph{Collected Papers}, Thales Verlag, pp.~166--171.

\bibitem{murty}
M.~Ram~Murty, \emph{Ramanujan series for arithmetic functions}, Hardy--Ramanujan J.~ {\bf 36} (2013), 21--33.

\bibitem{nz}
I.~Nivan and H.~S.~Zuckerman, \emph{An Introduction to the Theory of Numbers}, 4th ed, John Wiley \& Sons, New York, 1980.


\bibitem{rg}
H.~Rademacher and E.~Grosswald, \emph{Dedekind Sums}, Carus Math.~Monograph \#16, Math.~Assoc.~Amer., 1972.


\bibitem{Ram1918} S.~Ramanujan, \emph{On certain trigonometrical sums and their applications in the theory of numbers},  Trans. ~Cambridge
Philos.~ Soc. {\bf 22} (1918), 259--276.

\bibitem{cp}
S.~Ramanujan, \emph{Collected Papers of Srinivasa Ramanujan}, Cambridge Univ.~Press, 1927, 179--199; AMS Chelsea, Providence, 2000.

\bibitem{rassias}
M.~ T.~ Rassias and L\'aszl\'o T\'oth, \emph{Ramanujan Sums}, in \emph{Ramanujan: His Life, Legacy, and Mathematical Influence}, Springer, to appear.

\bibitem{palestine}
G.~Vinay, H.~T.~Shwetha, and K.~N.~Harshitha, \emph{Non-trivial trigonometric sums arising from some of Ramanujan theta function identities}, Palestine J.~Math.~Vol.~11(1) (2022), 130--134.

\bibitem{titchmarsh}
E.~C.~Titchmarsh, \emph{The Theory of the Riemann Zeta-Function}, Clarendon Press, Oxford, 1951.


\end{thebibliography}
\end{document}